\documentclass[12pt]{amsart}
\usepackage{amsmath}
\usepackage{hyperref}
\usepackage{amsxtra}
\usepackage[dvipsnames, x11names]{xcolor}
\usepackage{amsopn}
\usepackage{caption}
\usepackage[T1]{fontenc}
\usepackage[protrusion=true,expansion=true]{microtype}
\usepackage[pdftex]{graphicx}
\usepackage[in, headings]{fullpage}
\usepackage{url}
\usepackage{tabularray}
\usepackage{comment}
\usepackage[utf8]{inputenc}
\usepackage{lmodern}
\usepackage[main=english]{babel}
\usepackage[type={CC},modifier={by-nc-sa},version={3.0}]{doclicense}
\usepackage{amsthm}
\usepackage{amssymb}
\usepackage{amsfonts}
\usepackage{stmaryrd}
\usepackage{bbold}
\usepackage{mathrsfs}
\usepackage{fancyhdr}
\usepackage{fancyref}
\usepackage{float}
\usepackage{csquotes}
\usepackage{quiver}
\usepackage{mathtools}
\usepackage[style=alphabetic, mincitenames=1, maxnames=4, giveninits, backend=biber]{biblatex}
\usepackage{pdfcomment}

\AtEveryBibitem{
   \clearfield{month}
   \clearfield{series}
   \clearfield{venue}
   \clearname{editor}
   \clearlist{location} 
   \clearfield{doi}
   \clearfield{url}
   \clearfield{venue}
   \clearfield{issn}
   \clearfield{isbn}
   \clearfield{urldate}
   \clearfield{eventdate}
}

\newtheorem{numbering}{numbering}[section]
\numberwithin{numbering}{section}

\theoremstyle{plain}
\newtheorem{thm}[numbering]{Theorem}
\newtheorem{pps}[numbering]{Proposition}

\newtheorem{csq}[numbering]{Corollary}
\newtheorem{lem}[numbering]{Lemma}

\theoremstyle{definition}

\newtheorem{df}[numbering]{Definition}

\theoremstyle{remark}
\newtheorem{rmq}[numbering]{Remark}

\numberwithin{equation}{section}

\DeclareMathOperator{\identity}{id}

\DeclareMathOperator{\End}{End}

\DeclareMathOperator{\Ad}{Ad}
\DeclareMathOperator{\ad}{ad}

\DeclareMathOperator{\SO}{SO}

\DeclareMathOperator{\U}{U}
\DeclareMathOperator{\SU}{SU}

\DeclareMathOperator{\Spin}{Spin}
\DeclareMathOperator{\Sym}{Sym}

\DeclareMathOperator{\nv}{\nabla_{\mathbb{V}}}
\DeclareMathOperator{\nh}{\nabla_{\mathbb{H}}}
\DeclareMathOperator{\nhs}{\nabla_{\mathbb{H}^{\mg}}}
\DeclareMathOperator{\nhss}{\nabla_{\mathbb{H}^{\mg, s}}}
\DeclareMathOperator{\xx}{\mathbb{X}}
\DeclareMathOperator{\xs}{\mathbb{X}^{\mg}}
\DeclareMathOperator{\oo}{\mathbf{\Omega}}
\DeclareMathOperator{\doo}{d\mathbf{\Omega}}
\DeclareMathOperator{\ot}{\widetilde{\mathbf{\Omega}}}

\DeclareMathOperator{\mg}{\Omega}
\DeclareMathOperator{\dmg}{d\Omega}

\newcommand{\psm}{\pi_{SM}}
\newcommand{\pfm}{\pi_{FM}}
\newcommand{\dpsm}{d\pi_{SM}}
\newcommand{\dpfm}{d\pi_{FM}}
\newcommand{\rfm}{\mathbf{R}}
\newcommand{\rsfm}{\mathbf{R}^{\mg}}
\newcommand{\con}[1]{T_{#1}}

\newcommand{\N}{\mathbb{N}}
\newcommand{\Z}{\mathbb{Z}}

\newcommand{\R}{\mathbb{R}}
\newcommand{\C}{\mathbb{C}}
\newcommand{\HH}{\mathbb{H}}

\newcommand{\frk}[1]{\mathfrak{#1}}

\title[Magnetic flows in higher dimensions]{Tensor tomography and frame flow ergodicity for magnetic flows in higher dimensions}
\author{Louis-Brahim Beaufort}

\address{Louis-Brahim Beaufort, Université Paris-Saclay,  Laboratoire de mathématiques d'Orsay, 91405, Orsay, France}

\email{louis-brahim.beaufort@math.cnrs.fr}

\subjclass[2020]{37D30, 37D40, 53C15}

\keywords{Magnetic geodesics, Pestov identity, Magnetic frame flow, Frame flow ergodicity, Tensor tomography}

\begin{document}

\begin{abstract}
We extend two results from the theory of geodesic flows to the magnetic setting on manifolds of arbitrary dimension. First, we investigate the magnetic ray transform and establish a tensor tomography result. Second, we define and analyze the ergodicity of the magnetic frame flow under a pinching condition, building on work of Ceki\'{c}-Lefeuvre-Moroianu-Semmelmann.

These generalizations rely on new Pestov identities tailored to the magnetic flow, which extend and improve identities derived by Dairbekov-Paternain. 

In the process, we develop a framework that adapts several concepts of Riemannian geometry to the magnetic context, including covariant differentiation, torsion, curvature, and Jacobi fields. Notably, our curvature tensor generalizes the magnetic sectional curvature recently proposed by Assenza.
\end{abstract}

\maketitle

\section{Introduction} \label{sec:intro}

\subsection{Background}

The goal of this article is to generalize several results on geodesic flows over closed, negatively curved Riemannian manifolds to the magnetic setting. Given a closed Riemannian manifold $(M, g)$ and a differential $2$-form $\mg$ on $M$ (which is customarily closed), a magnetic geodesic $\gamma$ is a smooth curve $\gamma: I \to M$, where $I \subset \R$ is an interval, which satisfies the ordinary differential equation 
\begin{align*}
\nabla_{\gamma'}^{LC} \gamma' = \Omega \gamma',
\end{align*}
where $\nabla^{LC}$ denotes the Levi-Civita connection associated with $(M, g)$ and by a slight abuse of notation $\Omega$ denotes the endomorphism obtained from the $2$-form $\mg$ by metric duality. More precisely, if $p \in M$, $x, y \in T_p M$, we have $\mg(x, y) = \langle \Omega x, y \rangle$. The standard theory of ODE's shows that for every initial data $(p, v) \in TM$, there is a unique magnetic geodesic $\gamma$ defined over $\R$ such that $\gamma(0) = p$, $\gamma'(0) = v$. The magnetic flow $\varphi_t^{\mg}$, acting on the tangent bundle $TM$, is then defined by $\varphi_t^{\mg}(p, v) = (\gamma(t), \gamma'(t))$. In particular, when $\mg = 0$, $\varphi_t^0$ is the geodesic flow.

Geodesic flows and magnetic flows share several important properties. For instance, since magnetic trajectories have constant speed, they induce flows on the sphere bundles $S^r M = \{(p, v) \in TM, \Vert v \Vert = r\}$. Both flows preserve the standard volume (the Liouville volume form) on these bundles. Moreover, when $\mg$ is closed, both flows are Hamiltonian. There are still notable differences. Indeed, magnetic flows are in general not Reeb flows for any contact form. Also, nontrivial reparametrizations of magnetic geodesics are not magnetic geodesics, which implies that there is no affine connection on $M$ such that magnetic geodesics correspond to the geodesics for this connection.

Nevertheless, one can still study magnetic flows using the basic geometric framework associated with the Riemannian metric $g$ by introducing magnetic analogues of Riemannian objects, though this requires lengthy computations. Over surfaces, this approach turned out to be fruitful, enabling some authors to define a magnetic Gauss curvature and derive useful integral identities \cite{merry2011inverse}. Using these tools, some deep results \cite{assenza2024marked,echevarria2025rigidity} were generalized from the Riemannian to the magnetic setting. In higher dimensions, computations become significantly harder, and few results have been obtained until very recently, when a dynamically relevant notion of magnetic curvature appeared \cite{assenza2024magnetic,assenza2025hopf, terek2026magnetic} along with a notion of magnetic Jacobi fields, see also \cite{reber2026finiteness, beaufort2026rigidity} for applications. 

\medskip

A relevant notion of curvature is essential for deriving a magnetic variant of a classical Riemannian identity, called the Pestov identity. Introduced by Mukhometov \cite{mukhometov1975inverse}, this identity has played a central role in studying geometric inverse problems and spectral and dynamical rigidity \cite{guillemin1980inverse, croke1998rigidity, dairbekov2003integral, pestov2005rigid, paternain2013tomography, guillarmou2019length}. More recently, the Pestov identity was instrumental in obtaining sharper pinching estimates in the ergodicity of frame flows \cite{cekic2024frame,cekic2024unitary}. One can see this formula as a dynamically motivated variant of the Bochner/Weitzenböck identity \cite{cekic2025correspondence}.

A Pestov identity was established for magnetic surfaces and used successfully in the study of the magnetic ray transform in dimension 2, see \cite{ainsworth2013magnetic}. A natural question is therefore whether magnetic Pestov identities exist in higher dimensions and whether they can be applied to similar problems.

\subsection{Magnetic geometry}

Inspired by the recent derivation of the Pestov identity in \cite{cekic2025pestov}, we develop a magnetic differential calculus by lifting magnetic dynamics from the sphere bundle $SM$ to the orthonormal frame bundle $FM$. More precisely, we associate a Cartan connection on $FM \to SM$ to a distinguished lift $\Phi_t^{\mg}$ of the magnetic flow $\varphi_t^{\mg}$ on the unit tangent bundle $SM$. (Note: no prior knowledge of Cartan geometry is required to read this paper.) This structure represents a deformation of the classical Riemannian geometry induced by the magnetic field. The choice of the lift is nontrivial and can be interpreted in terms of torsion, see Remark \ref{rmq:choice_lift}; it is also related to the “covariant derivative” $\widetilde{\mathcal{D}}$ and the anisotropic Lorentz force $\widetilde{\Omega}$ introduced in \cite{assenza2025hopf}. In particular, we are able to show how this "covariant derivative" operator $\widetilde{\mathcal{D}}$ of \cite{assenza2025hopf}, which lacked a geometrical meaning, can be realized as a covariant derivative associated with our connection. We call this special lift $\Phi_t^{\mg}$ the \emph{magnetic frame flow}, and denote its generator $\xs$. We note that our definition does not coincide with the magnetic frame flow of \cite{reber2026finiteness}.

Using this framework, we geometrically derive notions of curvature, torsion, and Jacobi fields. The resulting curvature tensor $\mathcal{R}^{\mg}$, computed in Proposition \ref{pps:curvature}, is defined over the unit sphere bundle $SM$ and takes as input tangent vectors to $M$. Furthermore, $\mathcal{R}^{\mg}$ generalizes the magnetic sectional curvature operator from \cite{assenza2024magnetic}, which we call the \emph{tangential magnetic sectional curvature} and denote by $\mathcal{M}^{\mg}$. More precisely,
\begin{align*}
\mathcal{R}^{\mg}_{(p, v)}(x, v) v = \mathcal{M}^{\mg}_{(p, v)} x, \quad \forall p \in M, \forall v, x \in S_p M.
\end{align*}

Finally, our Jacobi fields correspond to those defined in \cite{assenza2025hopf}. Our work thus provides a firm geometric foundation for these concepts.

\medskip 

We note that there is a notion of magnetic torsion, which is never zero when $\mg \neq 0$ (see Lemma \ref{lem:torsion}). Thus, it is not surprising that the curvature tensor $\mathcal{R}^{\mg}$ lacks symmetries, which is exemplified by a nontrivial Bianchi identity (Proposition \ref{pps:bianchi}). As a result, we encounter unexpected computational difficulties when attempting to generalize several Riemannian arguments to the magnetic case.

In the particular case where $\mg$ is a constant multiple of an almost complex structure, most formulas simplify considerably. This is especially true for (constant multiples of) Kähler complex structures, yielding better results; see for instance Remark \ref{rmq:bianchi_kahler} and Theorem \ref{thm:ergo_kahler_frame_flow_intro}.

\medskip

We then establish several Pestov identities. First, following \cite{cekic2025pestov}, we prove a universal identity on the frame bundle. The choice of lift $\xs$ is crucial here in order to obtain a simple and useful formula, see Remarks \ref{rmq:choice_lift} and \ref{rmq:choice_lift_2}. Then, we deduce Pestov identities on tensor bundles over $SM$ constructed from the normal bundle via natural operations. This derivation extends the work of \cite{cekic2025pestov} and generalizes the Pestov identity on vector bundles with connection from \cite{guillarmou2016xray}. 

We also derive a localized Pestov identity for functions over the unit tangent bundle, extending \cite{guillarmou2016xray}. Although we limit ourselves to this case, our formalism applies equally well to tensor bundles, including those obtained from the normal bundle by natural operations such as $\mathcal{N}^*, \End(\mathcal{N}), \Sym^k \mathcal{N}, \Lambda^k \mathcal{N}$, and so on. For now, we only consider a natural connection over these bundles (derived from our magnetic connection) and not a general connection.

\medskip

We also emphasize that most of the results that we prove \emph{do not} rely on the assumption that $\mg$ is closed. Note that in dimension $2$, every $2$-form is closed. Our work thus suggests that similar identities may exist for more general families of variants of geodesic flows, as we discuss below.

\subsection{Tensor tomography}

As a first application, we study the tensor tomography problem under the assumption of negative tangential magnetic sectional curvature $\mathcal{M}^{\mg}$. When $\mg$ is closed and $\varphi_t^{\mg}$ is an Anosov flow (for instance, if $\mathcal{M}^{\mg}$ is negative-definite by \cite[Appendix A]{assenza2025hopf}), each nontrivial free homotopy class of curves $c$ in $M$ corresponds to a unique closed orbit $\gamma_c \in c$ of the magnetic flow. We also note that this fact remains true if $\mg$ is not closed and is sufficiently close to a closed magnetic $2$-form by structural stability of Anosov flows \cite{delallave1986perturbation}. Indeed, the two associated magnetic flows will be conjugate by a diffeomorphism isotopic to the identity; and conjugacies preserve closed orbits.

Let $\mathcal{C}$ denote the set of nontrivial free homotopy classes. The \emph{magnetic ray transform} $I^{\mg}$ is then defined by 
\begin{align*}
I^{\mg} f : c \in \mathcal{C} \mapsto \int_{\gamma_c} f(t) dt, \quad \forall f \in \mathcal{C}^\infty(SM).
\end{align*}

By Liv\v{s}ic theory \cite[Theorem 11.1.2]{lefeuvre2024microlocal}, the functions in the kernel of $I^{\mg}$ have the form $f = X^{\mg} u$ for $u \in \mathcal{C}^\infty(SM)$. The tensor tomography problem asks: if $f$ has Fourier degree $m \geq 0$ (meaning its fiberwise spherical harmonics expansion contains only terms of degree $\leq m$, see Section \ref{sec:spherical_harmonics}), what is the Fourier degree of $u$? When $u$ has necessarily degree $\leq \max(0, m - 1)$, we say that $I^{\mg}$ is \emph{solenoidal-injective on tensors of degree $\leq m$}.

\smallskip

This question has been the focus of a considerable amount of work, due to its connections with several geometric inverse problems such as boundary rigidity \cite{dairbekov2003integral, pestov2005rigid,paternain2013tomography,uhlmann2016inverse,guillarmou2017lens,guillarmou2017invariant,uhlmann2021rigidity} and marked length spectrum rigidity \cite{guillarmou2019length, guillarmou2025marked}. There, proving solenoidal injectivity of the ray transform (or more specifically, its restriction to $m$ tensors for $m \leq 2$) amounts to solving the linearized version of the geometric inverse problem, which has proven to be a stepping stone to the resolution of the nonlinear problem.

\smallskip

In the Riemannian setting, it is known \cite{paternain2015invariant} that for negatively curved metrics, $I^{\mg}$ is solenoidal injective on tensors of any order. The same result holds for Riemann surfaces with Anosov geodesic flow by \cite{paternain2014invariant, guillarmou2017invariant}.

The magnetic ray transform $I^{\mg}$ is known to be solenoidal injective on $m$-tensors in the following situations:
\begin{itemize}
\item When $n = 2$, $m \leq 2$, $\mg$ is closed and $\varphi_t^{\mg}$ is Anosov, by \cite[Theorem 1.1]{ainsworth2012magnetic}.
\item When $n \geq 3, m \leq 1$, $\mg$ is closed and $\varphi_t^{\mg}$ is Anosov, by \cite[Theorem B]{dairbekov2008optical}.
\item When $n \geq 3, m = 2$, $\mg$ is closed and the curvature satisfies a certain assumption \cite[Theorem 6.3]{munoz2025guillarmou}. More precisely, for $(p, v) \in SM$, denote
\begin{align*}
k_p(v) = \sup_{z \in S_p M, z \perp v} 2 \mathcal{R}(z, v, v, z) + \langle \Omega v, z \rangle^2 + (n+3) \Vert \Omega z \Vert^2 - 2 \langle (\nabla^{LC}_z \Omega) v, z \rangle.
\end{align*}
Then, if for any closed orbit $\gamma$ with period $T_\gamma$, we have 
\begin{align} \label{eq:crit_dp}
T_\gamma \int_{\gamma'} \max(0, k) \leq 4,
\end{align}
it follows that $I^{\mg}$ is solenoidal injective on tensors of order $\leq 2$.
\end{itemize}

Our first main result is the following:

\begin{thm} \label{thm:tensor_tomo_intro}
Let $(M, g)$ be a Riemannian manifold of dimension $n \geq 2$ and $\mg$ a differential $2$-form on $M$, with $\Omega$ also denoting the associated endomorphism. Let $m \geq 0$ and $\mathcal{M}^{\mg}$ denote the tangential magnetic sectional curvature, defined in Section \ref{sec:sec_curvature}.

\begin{itemize}
\item[(1)] Let $u \in \mathcal{C}^\infty(SM)$ be such that $X^{\mg} u$ has degree $0$. Then $u$ is constant on $SM$.
\item[(2)] For $m \geq 1$, assume that there exists $\kappa > 0$ such that we have the curvature bound
\begin{align*}
\langle \mathcal{M}^{\mg}_{(p, v)} z, z \rangle + C(m, n) \langle \Omega_p v, z \rangle^2 \leq - \kappa \Vert z \Vert^2, \quad \forall p \in M, \forall v \in S_p M, \forall z \in T_p M,
\end{align*}
where $C(m, n) = \frac{m(m-1)}{2m+n-2} + (n-2) \frac{m-1}{m}$. Then $I^{\mg}$ is solenoidal injective on tensors of degree $\leq m$.
\end{itemize}
\end{thm}

To our knowledge, Theorem \ref{thm:tensor_tomo_intro} is the first result on tensor tomography for tensors of order $m \geq 3$, and for non-closed magnetic forms.

\medskip

When $m \leq 1$, $C(m, n) \leq 0$, so it is enough to assume that $\mathcal{M}^{\mg}$ is negative definite. Thus, we generalize \cite[Theorem B]{dairbekov2008optical} to potentially non-closed magnetic $2$-forms, under a slightly stronger assumption of negative tangential magnetic curvature.

\smallskip

When $m = 2$, $C(m, n) = \frac{n}{2} - 1 + \frac{2}{n+2}$, and the assumption in Theorem \ref{thm:tensor_tomo_intro} can be rewritten in terms of the Riemann curvature tensor $\mathcal{R}$ of the metric $g$ as:

\begin{align*}
\mathcal{R}(z, v, v, z) - \langle (\nabla_z \Omega) v, z \rangle + \frac{1}{4} \Vert \Omega z \Vert^2 + \left ( \frac{n}{2} - \frac{1}{4} + \frac{2}{n+2} \right ) \langle \Omega v, z \rangle^2 < 0,
\end{align*}
for $v, z \in S_p M, p \in M$.

This criterion is simpler than \eqref{eq:crit_dp}, though it is difficult to determine which is sharper. Both may be used in applications. We mention as examples of applications stability estimates for $I^{\mg}$ obtained by Mu\~{n}oz-Thon and Richardson in \cite{munoz2025guillarmou} and local magnetic action rigidity \cite{beaufort2026rigidity} by the same authors in collaboration with the author of the present paper.

\smallskip

To prove Theorem \ref{thm:tensor_tomo_intro}, we use the first strategy from \cite{guillarmou2016xray} which relies on estimates on the tail of the spherical harmonics decomposition. This method gives suboptimal estimates. In the same paper, the authors give another, more efficient proof of tensor tomography which uses the finite Fourier degree property and the absence of conformal Killing tensors. We note that under negative tangential magnetic sectional curvature, no conformal Killing tensors exist (Corollary \ref{cor:conf_killing_tensors}), but we cannot establish finiteness of the Fourier degree; consequently, we cannot use this second approach.

\subsection{Magnetic frame flow ergodicity}

As a second application, we study ergodicity of the magnetic frame flow $\Phi_t^{\mg}$ with respect to a natural measure on the frame bundle $FM$, which is obtained by taking a product of the Riemannian volume on the base and a Haar measure on the fibers. 

In the Riemannian setting, the frame flow $\Phi_t^0$ over negatively curved manifolds has been extensively studied \cite{brin1974partially,brin1975extensions,brin1980ergodic,brin1982ergodic,brin1984frame, burns2003stable,cekic2024frame,cekic2024unitary,cekic2025pestov,cekic2024semiclassical} as the classical example of \emph{partially hyperbolic flow} \cite{hasselblatt2006partially}. The frame flow is also an example of \emph{isometric extension} of the geodesic flow, meaning that it commutes with the right action of the structure group $\SO(n-1)$ of the principal bundle $FM \to SM$.

Recall that given $\delta \in (0, 1)$, we say that a Riemannian manifold $(M^n, g)$ has \emph{negative, $\delta$-pinched sectional curvature} if there exists $K > 0$ such that $-K \leq \mathrm{sec}(P) \leq - K \delta$ for every $2$-plane $P$. Then, the Riemannian frame flow $\Phi_t^0$ is known to be ergodic if $n \neq 7$ is odd \cite{brin1980ergodic}, and for other dimensions in the following situations \cite{cekic2024frame}:
\begin{itemize}
\item If $n = 7$ and $(M, g)$ is $\sim 0.5$-pinched,
\item If $n = 4$ or $n = 2 \pmod{4}$, $n \neq 134$  and if $(M, g)$ is $\sim 0.3$-pinched,
\item If $n = 134$ or $n = 0 \pmod{4}$ and if $(M, g)$ is $\sim 0.6$-pinched.
\end{itemize}

\smallskip 

We now turn to the magnetic case. The magnetic frame flow $(\Phi_t^{\mg})_{t \in \R}$ is also an \emph{isometric extension} of the magnetic geodesic flow $(\varphi_t^{\mg})_{t \in \R}$. In particular, since $(\varphi_t^{\mg})_{t \in \R}$ preserves the standard volume (even when $\mg$ is not closed; this can be seen by a straightforward computation, decomposing the generator $X^{\mg}$ in horizontal and vertical parts), then $(\Phi_t^{\mg})_{t \in \R}$ also preserves the volume (this is also proved directly in Lemma \ref{lem:skew_adj}). We may thus inverstigate the ergodicity of $(\Phi_t^{\mg})_{t \in \R}$.

Denote, for $t > 0$, $\chi(t) = \frac{2\sqrt{t}}{3 + 2\sqrt{t}}$ and, for $n = \dim M \geq 3$, 
\begin{align*}
\delta^*(n) = 
\begin{cases*}
   0 \text{ if $n \neq 7$ is odd,} \\ 
   \chi(16) = \frac{8}{11} \text{ if $n = 7, 8$,} \\
   \chi(n-2) \text{ if $n \neq 134$ is even,} \\ 
   \chi(\max(132, \lambda_1(\SO(133) / (\mathbf{E}_7 / \Z_2)))) \text{ if $n = 134$}.
\end{cases*}
\end{align*}
Here, $\lambda_1(\SO(133) / (\mathbf{E}_7 / \Z_2))$ denotes the first nonzero eigenvalue of the Laplacian on the homogeneous space $\SO(133) / (\mathbf{E}_7 / \Z_2)$ for the metric induced by the metric $- \frac{1}{2} \mathrm{Tr}$ on $\frk{so}(133)$. This latter value may be computed in theory using standard tools of representation theory, although the very high dimension of this space makes it difficult. We expect that this exceptional value is irrelevant to our problem for the same reasons as in \cite{cekic2024frame,}, where similar difficulties arise.

Note that $\delta^*(n) \to 1$ as $n \to +\infty$ for even $n$.

\smallskip

Our second main result is the following:

\begin{thm} \label{thm:ergo_mag_frame_flow_intro}
Let $(M, g)$ be a negatively curved, $\delta$-pinched Riemannian manifold with $1 > \delta > \delta^*(n)$. 

There exists an \emph{explicit} constant $\epsilon = \epsilon(\delta)$ such that for any \emph{closed} differential $2$-form $\mg$ with $\Vert \mg \Vert_{\mathcal{C}^1} < \epsilon$, the magnetic frame flow $\Phi_t^{\mg}$ is ergodic.
\end{thm}

We give a more general statement in Theorem \ref{thm:ergo_mag_frame_flow}. For Kähler magnetic flows, nontrivial simplifications in the curvature tensor yields the following improvement:

\begin{thm} \label{thm:ergo_kahler_frame_flow_intro}
Let $(M, g, J)$ be a Kähler manifold and $\mg$ a constant multiple of $J$. 

If the magnetic sectional curvature $\mathrm{sec}^{\mg}$ is $\delta$-pinched for some $\delta \in (\delta^*(n), 1]$, then the magnetic frame flow $\Phi_t^{\mg}$ is ergodic.
\end{thm}

To our knowledge, this is the first explicit ergodicity result for the magnetic flow $\Phi_t^{\mg}$. While the Riemannian frame flow is stably ergodic in negative curvature \cite{burns2003stable}, which implies automatic ergodicity of the magnetic frame flow for sufficiently small magnetic fields, this approach yields no explicit bounds on when ergodicity holds.

We also point out that the assumption $\dmg = 0$ is only necessary in the proof in order to ensure that $\varphi_t^{\mg}$ is Anosov. Using the fact that small deformations of Anosov flows remain Anosov, one can use Theorem \ref{thm:ergo_mag_frame_flow} to obtain a result for non-closed $\mg$ that are small perturbations of a closed magnetic $2$-form.

\smallskip 

In principle, one could adapt the approach from \cite{cekic2024frame} using our machinery to obtain superior bounds (though still weaker than in \cite{cekic2024frame}), but this would require heavy computations. The proof we present is considerably shorter; we extend the approach of \cite{cekic2025pestov} (which applies only to real hyperbolic metrics).

\subsection{Perspectives}

We expect that our formalism can be applied to other generalizations of geodesic flows. The first examples that come to mind are thermostat flows, where the endomorphism $\Omega$ is allowed to depend on the variable $v \in SM$; equivalently, $\Omega$ is a section of $\Lambda^2(\psm^* TM)$. Indeed, a thermostat Pestov identity has been established over surfaces \cite{merry2011inverse} and used in several subsequent works, particularly in the study of Gaussian thermostats, which are the geodesic flows of affine metric connections (see for instance \cite{echevarria2025rigidity}). However, since thermostats are in general not volume preserving and taking into account the computational difficulties over surfaces, such a generalization would be nontrivial.

\medskip

This work also raises questions about the nature of the geometrical objects we introduce. Notably, it would be interesting to understand the holonomies associated with the Cartan connection, in order to understand the Brin group involved in the study of frame flow ergodicity. However, it is not clear what should be transported by the holonomies, and along which paths. For instance, the Cartan connection induces a nonlinear bundle connection on $SM \to M$ (already appearing in \cite{assenza2025hopf}), which remains mysterious.

\subsection{Organization of the paper}

Section \ref{sec:levi_civita} reviews classical Riemannian geometry and dynamics via the Levi-Civita connection. Section \ref{sec:magnetic_connection} introduces the magnetic connection, computes its torsion and curvature, and studies the associated sectional curvature by means of the Bianchi identity. Then, in Section \ref{sec:mag_pestov}, we state and prove the universal and associated magnetic Pestov identities. Frequency localization and the tensor tomography problem are treated in Section \ref{sec:freq_loc}. In Section \ref{sec:ergodicity}, we study ergodicity for the magnetic frame flow under very strong pinching. Magnetic Jacobi fields are discussed in Appendix \ref{sec:mag_jacobi}. Finally, Appendix \ref{sec:casimir} contains some material on representation theory used to give explicit pinching bounds in Theorems \ref{thm:ergo_mag_frame_flow_intro} and \ref{thm:ergo_kahler_frame_flow_intro}.

\medskip

\textbf{Acknowledgements.} 

This work is part of the author's PhD thesis. The author wishes to thank Thibault Lefeuvre and Andrei Moroianu for suggesting the problem as well as their careful guidance during the elaboration of this work. The author also wishes to thank Romain Dugast and Tristan Humbert for helping in finding an error in an early version of the paper.

This research was supported by the European Research Council (ERC) under the European Union's Horizon 2020 research and innovation programme (Grant agreement no. 101162990 -- ADG). 

\section{The Levi-Civita connection} \label{sec:levi_civita}

\subsection{Frame bundle geometry} \label{sec:levi_civita_fm}

Let \( (M, g) \) be a Riemannian manifold of dimension $n$. The frame bundle $\pfm : FM \to M$ is the principal $\SO(n)$-bundle of direct orthonormal frames of vector fields over $M$. Given a point $p \in M$, such a frame $(w_1, \ldots, w_n)$ of $T_p M$ may be identified with the unique isometry $w: (\R^n, g_{\R^n}) \to (T_p M, g_p)$, where $g_{\R^n}$ is the standard metric, which sends the canonical basis $(e_1, \ldots, e_n)$ on $(w_1, \ldots, w_n)$. An element $a \in \SO(n)$ then acts on $FM$ by right composition $R_a w = w \circ a$. 

\smallskip

The Levi-Civita connection associated with $g$ gives rise to a decomposition 
\begin{align}
TFM = \mathbb{H}^{FM} \oplus \mathbb{V}^{FM} \label{eq:connect_FM}
\end{align}
where $\mathbb{V}^{FM} = \ker \pfm$ is the \emph{vertical subspace} and $\mathbb{H}^{FM}$ is the \emph{horizontal subspace} associated with Levi-Civita. The vertical subspace is spanned by the \emph{fundamental vector fields} induced by the $\SO(n)$-action, which are defined for $\xi \in \frk{so}(n)$ by
\[ Y_\xi = \left . \frac{d}{dt} \right \vert_{t=0} R_{\exp(t\xi)}. \]
We thus obtain a basis of each fiber $\mathbb{V}^{FM}_w$ by choosing a basis $(\omega_\alpha)$ of $\frk{so}(n)$ and then considering ${(Y_{\omega_\alpha})}_\alpha$. 

On the other hand, the projection $\dpfm: TFM \to TM$ restricts to an isomorphism of vector bundles $\dpfm : \mathbb{H}^{FM} \to TM$. Given $z \in T_p M$, the \emph{horizontal lift} $z^{\HH}$ of $z$ is defined by 
\[ z^{\mathbb{H}} \in \mathbb{H}_w^{FM}, \quad \dpfm(z^{\HH}) = z.\]
This lets us define a natural basis of $\mathbb{H}_w^{FM}$ as follows. Given $x \in \R^n$, the \emph{standard vector field} associated with $x$ is $B_x = {w(x)}^\HH$. The standard basis $(e_1, \ldots, e_n)$ then induces the basis $(B_{e_1}, \ldots, B_{e_n})$ of $\mathbb{H}^{FM}_w$.

\medskip 

In the following, it will be useful to consider both $\frk{so}(n)$ and $\R^n$ (with the zero bracket) as subalgebras of the Lie algebra $\frk{g} \coloneq \frk{so}(n) \rtimes \R^n$ associated with the Lie group $\mathrm{Is}^+(\R^n)$ of (direct) \emph{affine} isometries of $\R^n$. The Lie bracket of $\xi + x, \eta + y \in \frk{g}$ is given by 
\begin{align}
   [\xi + x, \eta + y] = [\xi, \eta] + (\xi(y) - \eta(x)). \label{eq:semi_direct}
\end{align}
In particular, the bracket $[\xi, x] = \xi(x)$ coincides with the evaluation of skew-symmetric endomorphisms on vectors. We also have the useful identity 
\begin{align}
[\xi, x \wedge y] = \xi(x) \wedge y + x \wedge \xi(y). \label{eq:bracket_wedge}
\end{align}
We will take as inner product $\langle \xi + x, \xi' + x' \rangle = - \frac{1}{2} \mathrm{Tr}(\xi \cdot \xi') + \langle x, x' \rangle_{\R^n}$ where $\cdot$ denotes the composition of endomorphisms. With this normalization, the base formed by the $(e_1 \wedge e_j, e_k)_{1 \leq i < j \leq n, 1 \leq k \leq n}$ is orthonormal. Moreover, we have the relations
\begin{align}
\langle [\xi, \eta], \eta' \rangle = - \langle \eta, [\xi, \eta'] \rangle, \quad \forall \xi, \eta, \eta' \in \frk{so}(n), \label{eq:bi_inv} \\
\langle \xi x, y \rangle = \langle \xi, x \wedge y \rangle \quad \forall \xi \in \frk{so}(n), \forall x, y \in \R^n. \label{eq:scalar_wedge}
\end{align}

\smallskip 

The two bases $(Y_{\omega_\alpha}), (B_{e_j})$ jointly trivialize $TFM$. They also allow us to define a natural Riemannian metric $g_{FM}$ on $FM$, such that ${(Y_{\omega_\alpha}, B_{e_j})}_{\alpha, j}$ becomes an orthonormal basis. In this way, the projection $\pfm: FM \to M$ becomes a Riemannian submersion. The associated volume form $\mathrm{vol}_{FM}$ is called the \emph{Liouville volume form}. We shall denote $\mu = \mathrm{vol}_{FM} / \mathrm{Haar}(\SO(n-1))$ the renormalized volume, so that fibers of $FM \to SM$ have volume $1$.

\subsection{Structural equations} \label{sec:levi_civita:struct}

Recall that the Riemannian curvature tensor $\mathcal{R}$ of $g$ is defined as the $(1, 3)$ tensor over $M$ by the formula $\mathcal{R}(X, Y) \coloneq [\nabla_X, \nabla_Y] - \nabla_{[X, Y]}$, where $\nabla$ is the covariant derivative associated with the Levi-Civita connection. This tensor has additional symmetries, and thus defines a (symmetric) endomorphism of $\Lambda^2 TM$ through the metric:
\begin{align}
   \langle \mathcal{R}(X \wedge Y), Z \wedge W \rangle = \langle \mathcal{R}(X, Y) Z, W \rangle.
\end{align}
Equivalently, one may see $\mathcal{R}$ as a $\SO(n)$-equivariant function $\rfm: FM \to \End(\frk{so}(n))$ by the rule 
\begin{align}
   \rfm_w(\xi) = w^{-1} \mathcal{R}_p(w(\xi)).
\end{align}
Here, we still denote $w$ the isometry $\frk{so}(n) \to \Lambda^2 T_p M$ induced by the isometry $w: \R^n \to T_p M$.

In the following, when given functions $\xi: FM \to \frk{so}(n)$ and $x: FM \to \R^n$, we will simply denote 
\begin{align}
Y_\xi = \sum_\alpha \langle \xi, \omega_\alpha \rangle Y_{\omega_\alpha} \in \mathcal{C}^\infty(FM, \mathbb{V}^{FM}), \quad B_x = \sum_{i=1}^n \langle x, e_i \rangle B_{e_i} \in \mathcal{C}^\infty(FM, \mathbb{H}^{FM}). \label{eq:notation}
\end{align}

The commutation relations between fundamental and standard fields are given by the following \emph{structural equations}:
\begin{lem}[{\cite[Lemma 2.1]{cekic2025pestov}}] 
For all $\xi, \xi' \in \frk{so}(n), x, x' \in \R^n$, we have 
\begin{align}
   [Y_{\xi}, Y_{\xi'}] = Y_{[\xi, \xi']}, \quad [Y_\xi, B_x] = B_{\xi(x)}, \quad [B_x, B_{x'}] = - Y_{\rfm(x \wedge x')}. \label{eq:eq_struct}
\end{align}
\end{lem}

\subsection{Geodesic and frame flows} \label{sec:geod_frame_flows}

Let $\psm: SM \to M$ denote the unit tangent bundle of $M$, defined by 
\begin{align*}
   SM = \{(p, v) \in TM, \Vert v \Vert^2 = 1 \}, \quad \psm(p, v) = p.
\end{align*}
The \emph{geodesic flow} $(\varphi_t)_{t \in \R}$ on $SM$ is defined as follows. Let $(p, v) \in SM$, and $\gamma$ be the unique geodesic for the metric $g$ such that $\gamma(0) = p, \gamma'(0) = v$. Then, by definition, $\varphi_t(p, v) = (\gamma(t), \gamma'(t)) \in SM$. We denote $X$ the generator of $(\varphi_t)_{t \in \R}$.

The frame flow $(\Phi_t)_{t \in \R}$ is then defined as the flow of parallel transport along geodesics. More precisely, given a frame $w \in F_pM$ and $v = w(e_1)$, denote $\tau_t: T_p M \to T_{\gamma(t)} M$ the parallel transport of vector fields along $\gamma \vert_{[0, t]}$, where $\gamma$ is defined by $v$ as before. Then, we define 
\begin{align}
\Phi_t(w) \coloneq \tau_t \circ w.
\end{align}
Equivalently, the components of the frame $\Phi_t(w) \in FM$ are given by 
\begin{align*}
\Phi_t(w) = (\gamma'(t), \tau_t w(e_2), \ldots, \tau_t w(e_n)).
\end{align*}

Remark that the frame bundle $FM$ fibers over $SM$ naturally, by $\pi: w \in FM \mapsto w(e_1) \in SM$. If we denote $\SO(n-1)$ the stabilizer of $e_1$ in $\SO(n)$, then one sees that $\SO(n-1)$ acts simply transitively on the fibers of $\pi: FM \to SM$, and thus defines a principal fibration.

This lets us see $\Phi_t$ as an extension of $\varphi_t$: indeed, we have $\pi \circ \Phi_t = \varphi_t \circ \pi$. The flow $\Phi_t$ is moreover invariant under the right action of $\SO(n-1)$, in the sense that 
\begin{align*}
R_a \circ \Phi_t(w) = \tau_t \circ w \circ a  = \Phi_t \circ R_a(w).
\end{align*}

The generator of the frame flow is denoted $\xx$. Note that by definition of parallel transport, we have $\xx = B_{e_1}$.

\subsection{Induced structures on the unit tangent bundle} \label{sec:induced_sm}

We now explain how the structures on the principal bundle $FM \to M$ induce structures on the unit tangent bundle $SM$. 

We first review the classical theory of the bundle $SM \to M$. There is a decomposition
\[ TSM = \mathbb{H}^{SM, tot} \oplus \mathbb{V}^{SM} \] 
obtained by projecting the corresponding decomposition \eqref{eq:connect_FM} on the frame bundle. The vertical space $\mathbb{V}^{SM}$ is the kernel of the projection $\psm : SM \to M$, and the total horizontal subspace $\mathbb{H}^{SM}$ can be described as the kernel of the connection map $\mathcal{K}: TSM \to TM$, which is defined by 
\[ \mathcal{K}_{(p, v)} \zeta = \left . \frac{d}{dt} \right \vert_{t=0} PT_{p(t) \to p(0)} v(t). \]
Here, for $(p, v) \in SM$, $\zeta \in T_{(p, v)} SM$, we choose a path $z(t) = (p(t), v(t))$ such that $z'(0) = \zeta$ and denote $PT_{p(t) \to p(0)} : T_{p(t)} M \to T_{p(0)} M$ the parallel transport map along $t \mapsto p(t)$ associated with Levi-Civita. Note that since $t \mapsto v(t)$ has constant norm $1$, the vector $\mathcal{K}_{(p, v)} \zeta \in v^\perp \subset T_p M$. As a result, $\mathcal{K}$ realizes an isomorphism between the vertical bundle $\mathbb{V}^{SM} \to SM$ and the normal bundle $\mathcal{N} \to SM$, which is defined by $\mathcal{N}_{p, v} = v^\perp \subset T_p M$. 

The total horizontal subspace $\mathbb{H}^{SM, tot}$ further splits as $\mathbb{H}^{SM, tot} = \R X \oplus \mathbb{H}^{SM}$, where $X$ is the geodesic vector field as before. The projection then restricts to an isomorphism $\dpsm: \mathbb{H}_{(p, v)} \to v^\perp \subset T_p M$, which induces an isomorphism $\mathbb{H} \cong \mathcal{N}$ with the normal bundle. 

Associated to these isomorphisms, one can define the \emph{vertical lift} $z^\vee \in \mathbb{V}^{SM}_v$ and the \emph{horizontal lift} $z^H \in \mathbb{H}^{SM}_v$ of a normal vector $z \in \mathcal{N}_v$ by 
\[ \mathcal{K}(z^\vee) = v, \quad \dpsm(z^H) = z. \]

\bigskip 

We now turn to the fibration $\pi: FM \to SM$, defined by $v \coloneq \pi(w) = w(e_1)$ for $w \in FM$. The fiber of $FM \to SM$ at $v$ is acted on transitively by the stabilizer $\SO(n-1)$ of $e_1 \in \R^n$ for the standard action of $\SO(n)$ of $\R^n$. Hence, $FM \to SM$ is a principal $\SO(n-1)$ fibration. 

One can identify the fibration $FM \to SM$ with the \emph{frame bundle} $\mathrm{Fr}(\mathcal{N}) \to SM$ associated with the vector bundle $\mathcal{N} \to SM$, which is defined by 
\begin{align}
\mathrm{Fr}_v(\mathcal{N}) = \{ \text{isometries } w: (\R^{n-1}, g_{\R^{n-1}}) \to \mathcal{N}_v\}, \quad \forall v \in SM.
\end{align}
The identification $FM \cong \mathrm{Fr}(\mathcal{N})$ is simply the restriction $w \mapsto w\vert_{e_1^\perp}$. 

There is a parallel transport induced on $FM \to SM$ which can be understood in two ways. One can define a principal connection on $FM \to SM$ by taking as horizontal space the distribution $\mathbb{H}^{FM} \oplus \mathrm{span}\langle Y_{e_1 \wedge y}, y \in \R^n \rangle$. Alternatively, one can also see parallel transport of frames of $\mathcal{N}$ as induced by parallel transports of sections of $\mathcal{N}$ for the pullback connection $\psm^* \nabla^{LC}$.

\medskip

We also remark that this notion of parallel transport can be upgraded to a Cartan connection on $FM \to SM$. We refer the reader to \cite{sharpe2000differential} for a general exposition; our aim here is only to point out that the correct formalism to adopt here is that of Cartan connections.

\begin{df}[Cartan connection]
Let $N$ be a manifold, $G$ a connected Lie group and $H$ a closed, connected Lie subgroup of $G$. We denote $\frk{g}, \frk{h}$ their Lie algebras. A Cartan connection on $N$ modeled on the pair $(G, H)$ is the data of:
\begin{itemize}
\item A principal bundle $P \to N$ with structure group $H$, 
\item A $\frk{g}$-valued 1-form $\theta$,
\end{itemize}
such that:
\begin{itemize}
\item We have $H$-equivariance: $\theta_{w \cdot a} = \Ad(a)^{-1} \theta_w$ for $w \in P, a \in H$ (here, $\Ad$ is the adjoint representation of $H$ on $\frk{g}$).
\item For $\xi \in \frk{h}$, the fundamental field $Y_\xi = \frac{d}{dt} \vert_{t = 0} R_{\exp(t \xi)}$ satisfies $\theta(Y_\xi) = \xi$.
\item The form $\theta$ is pointwise an isomorphism: $\theta_w: T_w P \to \frk{g}$ is bijective. Equivalently, $\theta$ induces an isomorphism $TP \cong \frk{g} \times P$ of bundles over $P$.
\end{itemize}
\end{df}

In our context, we take $G$ to be the group of affine isometries of $\R^n$, $H = \SO(n-1)$ the stabilizer of $e_1 \in \R^n$, $N = SM$, $P = FM$. The Lie algebra $\frk{g}$ splits as a semi-direct product $\frk{g} = \frk{so}(n) \rtimes \R^n$ (see \eqref{eq:semi_direct}), and the $1$-form $\theta$ is defined by 
\[ \theta(Y_\xi) = \xi \in \frk{so}(n) \subset \frk{g}, \quad \theta(B_x) = x \in \R^n \subset \frk{g} \] 
for $\xi \in \frk{so}(n), x \in \R^n$. 

We do not wish here to give a comprehensive overview of the theory of Cartan geometries. Thus, we shall only mention that by “forgetting” the individual directions of the standard fields $B_x$ and $Y_{e_1 \wedge y}$, the Cartan connection induces a principal connection on $FM \to SM$, which is exactly the principal connection mentioned above. In particular, it induces the pullback connection $\psm^* \nabla^{LC}$ on the associated bundle $\mathcal{N}$.

\subsection{The Pestov identity on the frame bundle} \label{sec:levi_civita_pestov}

We now study the action of $\xx \coloneq B_{e_1}$ on functions over $FM$. Given $f \in \mathcal{C}^\infty(FM)$, we define its \emph{vertical gradient} $\nv f \in \mathcal{C}^\infty(FM, \frk{so}(n))$ by 
\begin{align}
\nv f = \sum_\alpha Y_{\omega_\alpha} f \ \omega_\alpha.
\end{align}
We then define the \emph{horizontal gradient} $\nh f \in \mathcal{C}^\infty(FM, \R^n)$ by 
\begin{align}
\nh f = \sum_{j=2}^n B_{e_j} f \ e_j.
\end{align}
The structure equations \eqref{eq:eq_struct} rewrite 
\begin{lem}[{\cite[Lemma 2.1]{cekic2025pestov}}]
Let $f \in \mathcal{C}^\infty(FM)$. For $x \in \R^n$, we denote $x^\perp = x - \langle x, e_1 \rangle e_1$ the orthogonal projection of $x$ on $e_1^\perp$. We have
\begin{align}
[\xx, \nv] f = - e_1 \wedge \nh f, \quad [\xx, \nh] f = - [\rfm(\nv f) e_1]^\perp.
\end{align}
\end{lem}

The fundamental fields $Y_{\omega_\alpha}$ and standard fields $B_{e_i}$ preserve the Liouville measure $\mu$, hence are skew-adjoint for the $L^2(\mu)$ inner-product on $\mathcal{C}^\infty(FM)$. We deduce expressions for the formal adjoints of $\nv, \nh$: 
\begin{align}
\nv^* \xi &= - \sum_\alpha Y_{\omega_\alpha} \langle \xi, \omega_\alpha \rangle \quad \forall \xi \in \mathcal{C}^\infty(FM, \frk{so}(n)), \\ 
\nh^* x &= - \sum_{j=2}^n B_{e_j} \langle x, e_j \rangle \quad \forall x \in \mathcal{C}^\infty(FM, \R^n).
\end{align}

Using these relations, one can prove the following identity, called the universal Pestov identity:

\begin{thm}[{\cite[Proposition 3.2]{cekic2025pestov}}]
Let $f \in \mathcal{C}^\infty(FM)$. We have the following:
\begin{align}
\xx^* \nv^* \nv \xx f = \nv^* \xx^* \xx \nv f + (n-1) \xx^* \xx f + \nv^* (e_1 \wedge \rfm(\nv f) e_1).
\end{align}
\end{thm}

\subsection{Tensor bundles and associated Pestov identities} \label{sec:levi_civita_pestov_tensor}

We now explain how to prove Pestov identities for bundles related to the frame bundle $FM$ such as bundles of tensors built from normal vectors.

Let $\rho: \SO(n-1) \to E$ be a unitary representation of $\SO(n-1)$, the associated bundle $EM \to SM$ is defined by $EM \coloneq FM \times_\rho E = FM \times E / \sim$, where $(w, z) \sim (w \cdot a, \rho(a)^{-1}z)$ for every $a \in \SO(n-1)$. We denote $[w; z]$ the equivalence class of $(w, z)$.

Remark that applied to the standard representation $\rho: \SO(n-1) \hookrightarrow \SO(n) = \SO(\R^{n})$, the associated bundle $\R^{n} M$ is naturally isomorphic over $SM$ to the pullback bundle $\psm^* TM \to SM$ through the map $[w; z] \mapsto w(z)$. This bundle can be concretely described as the sum $\psm^* TM = \R v \oplus \mathcal{N}$ where $v$ is the tautological section and $\mathcal{N}_v = v^\perp$ is the normal bundle. In particular, sections of $\mathcal{N}$ are sections of an associated bundle, and similarly sections of tensor bundles such as $\mathcal{N}^*, \End(\mathcal{N}), \Lambda^2 \mathcal{N}$, etc. will also be sections of an associated bundle for the corresponding representation of $\SO(n)$ into $(\R^n)^*, \End(\R^n), \Lambda^2 \R^n$, etc. 

Let now $s$ be a section of the associated bundle $EM$. Then, one associates to it a $\SO(n-1)$-equivariant function $\mathbf{s}: FM \to E$ by the rule 
\[ s_v = [w; \mathbf{s}_w], \quad \forall w \in FM, v = w(e_1). \]
If we let $(f_1, \ldots, f_r)$ be a basis of $E$, then the vector fields $Y_\xi, B_x$ act on functions $FM \to E$ by acting on each component in the basis (and this definition does not depend on the choice of the basis). This allows us to define vertical and horizontal gradients as follows:
\begin{align}
\nv s \vert_v &= \sum_{j=2}^n [w; Y_{e_1 \wedge e_j} \mathbf{s} \vert_w] \otimes w(e_j) \in EM \otimes \mathcal{N}, \\ 
\quad \nh s \vert_v &= \sum_{j=2}^n [w; B_{e_j} f \vert_w] \otimes w(e_j) \in EM \otimes \mathcal{N}.
\end{align} 

In the other directions, the derivatives may be computed using the equivariance of $\mathbf{s}$ by the formula 
\begin{align}
Y_\xi \mathbf{s} = - (\rho_* \xi) \mathbf{s}
\end{align}
where $\rho_*: \frk{so}(n) \to \End(V)$ is the differential of $\rho$ at the identity. As a result, we have the formula:
\begin{align*}
\nv \mathbf{s} = \sum_{j=2}^n Y_{e_1 \wedge e_j} \mathbf{s} \otimes e_1 \wedge e_j - \sum_{\omega_\alpha e_1 = 0} (\rho_* \omega_\alpha) \mathbf{s} \otimes \omega_\alpha.
\end{align*}
Here, we use the shorthand notation $\sum_{\omega_\alpha e_1 = 0}$ to denote the sum on the basis vectors $e_i \wedge e_j, 2 \leq i < j \leq n$ which span $\frk{so}(n-1)$. Notice that the first sum corresponds to the vertical gradient of $s$. We may then apply our structural equations componentwise in any basis $(f_1, \ldots, f_r)$ of $E$ to obtain the following:

\begin{pps}[Vector-valued structural equations] \label{pps:vec_struct}
Let $s$ be a section of an associated bundle $EM = FM \times_\rho E$ and $\mathbf{s}: FM \to E$ the associated $\SO(n-1)$-equivariant function. We have the relations
\begin{align}
[\xx, \nv] \mathbf{s} &= - e_1 \wedge \nh \mathbf{s}, \label{eq:struct_vec_1} \\
[\xx, \nh] \mathbf{s} &= - (\identity_E \otimes [\rfm(e_1, \cdot) e_1]^\perp) (\nv \mathbf{s})e_1 + \mathbf{F}^E(\mathbf{s}), \label{eq:struct_vec_2} \\ 
\nv^* \nh - \nh^* \nv &= (n-1) \xx. \label{eq:struct_vec_3}
\end{align}
Here, the function $\mathbf{F}^E: FM \to E^* \otimes E \otimes \mathcal{N}$ is given by 
\begin{align} \label{eq:extra_curv_term}
\mathbf{F}^E(z) = \sum_{\omega_\alpha e_1 = 0} (\rho_* \omega_\alpha) z \otimes [\rfm(\omega_\alpha) e_1]^\perp.
\end{align}
\end{pps}

In much the same manner as \cite[Corollary 3.6]{cekic2025pestov}, we obtain an associated Pestov identity on the fibration $EM \to SM$:
\begin{thm} \label{thm:tensor_pestov}
Let $s$ be a section of an associated bundle $EM = FM \times_\rho E$. We have 
\begin{align}
\Vert \nv \xx s \Vert^2 = \Vert \xx \nv s \Vert^2 + (n-1) \Vert \xx s \Vert^2 - \langle R \nv s, \nv s \rangle - \langle F^E(s), \nv s \rangle.
\end{align}
Here, $R \in \mathcal{C}^\infty(SM, \End(EM \otimes \mathcal{N}))$ is defined by $R_v(z) = \identity_E \otimes \mathcal{R}(z, v) v$, and similarly $F^E \in \mathcal{C}^\infty(SM, \End(EM \otimes \mathcal{N}))$ is the tensor on $SM$ associated with the function $\mathbf{F}^E$ defined in \eqref{eq:extra_curv_term}.
\end{thm}
\begin{proof}
Given a section $s$ of $EM$; we apply the proof of \cite[Theorem 3.1]{cekic2025pestov} componentwise to the associated function $\mathbf{s}: FM \to E$, using the above structural equations. We obtain 
\begin{align*}
&\xx^* \nv^* \nv \xx \mathbf{s} - \nv^* \xx^* \xx \nv \mathbf{s} \\
&\quad = (n-1) \xx^* \xx \mathbf{s} + \nv^* (\identity_E \otimes e_1 \wedge \rfm(e_1, \cdot) e_1) \nv \mathbf{s} \cdot e_1 - \nv^* \mathbf{F}^E \mathbf{s}.
\end{align*}
Then, taking the inner product with $\mathbf{s}$, we deduce 
\begin{align*}
&\Vert \nv \xx \mathbf{s} \Vert^2 - \Vert \xx \nv \mathbf{s} \Vert^2 \\
&\quad = (n-1) \Vert \xx \mathbf{s} \Vert^2 + \langle (\identity_E \otimes e_1 \wedge \rfm(e_1, \cdot) e_1) \nv \mathbf{s} \cdot e_1, \nv \mathbf{s} \rangle - \langle \mathbf{F}^E \mathbf{s}, \nv \mathbf{s} \rangle.
\end{align*}
We then remark that every integrand defining the norms and inner products is $\SO(n-1)$-invariant, and thus, integrating on the fibers of $FM \times E \to EM$ and using Fubini's theorem gives 
\begin{align}
\Vert \nv \xx s \Vert^2 = \Vert \xx \nv s \Vert^2 + (n-1) \Vert \xx s \Vert^2 + \langle \mathcal{R}(v, \nv s)v, \nv s \rangle - \langle F^E(s), \nv s \rangle.
\end{align}
The result follows.
\end{proof}

\section{The magnetic connection} \label{sec:magnetic_connection}

\subsection{Magnetic dynamics} \label{sec:mag_dyn}

Let $(M, g)$ be a Riemannian manifold and $\mg \in \mathcal{C}^\infty(M, \Lambda^2 T^* M)$ a $2$-form. Customarily, one asks moreover that $\mg$ is closed; however most of our results do not need this assumption, and thus we shall consider the general case. We shall also denote $\Omega \in \mathcal{C}^\infty(M, \End(TM))$ the skew-symmetric endomorphism associated with $\mg$ through $g$, called the \emph{Lorentz force} and defined by 
\[\Omega(x, y) = g(\Omega x, y). \]

\begin{df} \label{def:mag_geodesic}
A \emph{magnetic geodesic} for the system $(M, g, \mg)$ is a smooth curve $\gamma: I \to M$, where $I \subset \R$ is an open interval, such that the \emph{magnetic geodesic equation} is satisfied:
\begin{align}
\nabla_{\gamma'(t)} \gamma'(t) = \Omega \gamma'(t), \quad \forall t \in I.
\end{align}
\end{df}

This equation is a nonlinear ordinary differential equation of order 2, and thus for any initial data $(p, v) \in TM$, there is a unique maximal magnetic geodesic $\gamma_{p, v}$ such that $\gamma_{p, v}(0) = p, \gamma_{p, v}'(0) = v$. As a result, one can define a flow $\varphi_t^{\mg}$ on $TM$ by $\varphi_t^{\mg}(p, v) = (\gamma_{p, v}(t), \gamma_{p, v}'(t))$, which we call the \emph{magnetic geodesic flow}. Its generator $X^{\mg}$ is given by
\[ X^{\mg} = X + (\Omega v)^\vee, \] 
where $X$ is the geodesic vector field and $(\Omega v)^\vee \in \mathbb{V}_v^{SM}$ is the vertical lift of $\Omega v \in T_p M$.

It is easy to see, using the skew-symmetry of $\Omega$, that magnetic geodesics have constant speed; thus, the flow $\varphi_t^{\mg}$ preserves the sphere bundles $S_r M = \{ v \in TM, \Vert v \Vert = r \}, r > 0$. In particular, if $M$ is closed, then $\varphi_t^{\mg}$ is complete.

For the geodesic flow $\varphi$, all the restrictions $\varphi \vert_{S_r M}$ are reparametrizations of each other. This is no longer true for magnetic flows $\varphi_t^{\mg}$, where the behaviors on $S_r M$ for $r$ small and $r$ large will be typically very different. Note that to study a specific energy level $r$, we can rescale the magnetic field $\mg$ by $\frac{1}{r}$ and look at the flow $\varphi_t^{\mg/r}$ on $SM$. Thus, in the following, we will always work on $SM$, or equivalently with $r = 1$.

\smallskip 

When $\mg$ is closed, the flow $\varphi_t^{\mg}$ can be seen as a Hamiltonian flow associated with the Hamiltonian $\frac{1}{2} \Vert v \Vert^2$ on $TM$ and the symplectic form $\omega^{\mg}$ on $TM$ defined by 
\[ \omega^{\mg}(\xi, \eta) = \langle \mathcal{K}(\xi), {\psm}_* \eta \rangle - \langle {\psm}_* \xi, \mathcal{K}(\eta) \rangle - \Omega ({\psm}_* \xi, {\psm}_* \eta), \]
where ${\psm}_*$ is the pushforward of vector fields.

\subsection{The magnetic frame flow} \label{sec:mag_frame_flow}

We now want to lift the generator $X^{\mg} = X + (\Omega v)^\vee$ to the frame bundle $FM$. First, we lift the tensor $\Omega$ to a function which we denote $\oo: FM \to \Lambda^2 \R^n$.

Then, there are at least two natural choices of lifts of $X^{\mg}$, which are $X + Y_{\oo}$ and $X + Y_{e_1 \wedge \oo e_1}$ (see Equation \eqref{eq:notation} for the formal definition in terms of basis vectors). Further, any affine combination of these will also be a lift. 

Computations, though, show that there \emph{is} a preferred lift $\xs$ which is coincidentally the midpoint of the two previous lifts (see Remark \ref{rmq:choice_lift}). We denote 
\begin{align}
\ot = \frac{1}{2} e_1 \wedge \oo e_1 + \frac{1}{2} \oo = e_1 \wedge \oo e_1 + \oo^0, \label{def:ot}
\end{align}
where $\oo^0 = \frac{1}{2} (\oo - e_1 \wedge \oo e_1)$. The decomposition $\ot = e_1 \wedge \oo e_1 + \oo^0$ corresponds to the orthogonal direct sum $\frk{so}(n) = \frk{so}(n-1) \oplus e_1 \wedge \R^{n-1}$. 

Notice that $\ot, e_1 \wedge \oo e_1, \oo^0$ are $\SO(n-1)$-equivariant functions $FM \to \frk{so}(n)$, and thus define tensors $\widetilde{\Omega}, v \wedge \Omega v, \Omega^0$ over $SM$; or, more precisely, sections of $\Lambda^2 \psm^* TM \to SM$. However, they are not $\SO(n)$-equivariant, as they depend on $e_1$. 

Recall that $\oo$ is derived from the section $\Omega$ of $\End(TM) \to M$, hence is $\SO(n)$-equivariant. More precisely, the equivariance writes
\begin{align*}
\oo_{w \cdot a} = \Ad(a)^{-1} \oo_w =  a^{-1} \circ \oo_w \circ a \quad \forall a \in \SO(n).
\end{align*}
By differentiating this, we obtain the useful formula 
\begin{align} \label{eq:diff_omega}
Y_\xi \oo = [\oo, \xi], \quad \forall \xi \in \frk{so}(n).
\end{align}
With the Leibniz rule, one can then easily differentiate $e_1 \wedge \oo e_1, \ot, \oo^0$.

\medskip 

We can now define:
\begin{df}
Let $\xs \coloneq \xx + Y_{\ot}$. The \emph{magnetic frame flow} $\Phi_t^{\mg}$ is defined to be the flow of the vector field $\xs$.
\end{df}
By definition, $\xs$ projects to $X^{\mg}$ on $SM$, which lets us see the flow $\Phi_t^{\mg}$ as an extension of the magnetic flow $\phi_t^{\mg}$.

\medskip

The evolution of the magnetic frame flow may be understood in terms of the following ODE. Let $(p, v) \in SM$, denote $\gamma$ the magnetic geodesic starting at $p$ in the direction $v$, which solves 
\begin{align*}
\nabla^{LC}_{\gamma'} \gamma' = \Omega \gamma'.
\end{align*}
Let $w \in F_pM$ be a frame such that $w(e_1) = v$; then, we may write 
\begin{align*}
\Phi_t^{\mg}(w) = (\gamma(t), \gamma'(t), w_2(t), \ldots, w_n(t)),
\end{align*}
where the components $w_i(t), i \geq 2$, which are orthogonal to $\gamma'(t)$, satisfy 
\begin{align*}
(\psm^*\nabla^{LC})_{X^{\mg}} w_i(t) - \Omega^0 w_i(t) = 0.
\end{align*}
This is seen by making the operator $\xs = \xx + Y_{\ot} = \xx + Y_{e_1 \wedge \oo e_1} + Y_{\oo^0}$ act on sections $s$ of the normal bundle through their lift $\mathbf{s}: FM \to \R^{n-1}$ as explained in Sections \ref{sec:induced_sm} and \ref{sec:levi_civita_pestov_tensor}. 

\begin{df} \label{def:cov_diff}
Given a magnetic geodesic $\gamma$, we obtain a covariant differentiation along magnetic geodesics $\nabla^{\mg}_t$ acting on sections of $\mathcal{N} \to SM$ along $\gamma'$, which is given by 
\begin{align*}
\nabla^{\mg}_t s = (\psm^*\nabla^{LC})_{X^{\mg}} s - \Omega^0 s.
\end{align*}
\end{df}

\begin{rmq}
A similar operator denoted $\tilde{\mathcal{D}}$ was also considered in \cite{assenza2025hopf}, although it acts on a different bundle.
\end{rmq}

\medskip 

We note that the flow $\Phi_t^{\mg}$ commutes with the right action of $\SO(n-1)$ on the principal bundle $\pi: FM \to SM$. As a result, $\Phi_t^{\mg}$ is an \emph{isometric extension} of $\phi_t^{\mg}$, meaning that for every $v \in SM$, $t \in \R$, the map $\Phi_t^{\mg}: \pi^{-1}(v) \to \pi^{-1}(\phi_t^{\mg} (v))$ is an isometry. Indeed, this map is a $\SO(n-1)$-equivariant bijection between the two fibers, which are copies of $\SO(n-1)$.

\subsection{Magnetic standard fields} \label{sec:mag_std_fields}

By analogy with the Riemannian theory, we now define “magnetic standard vector fields” as follows.
\begin{df} \label{def:mag_std}
We denote 
\begin{align*}
B^{\mg}_{e_1} \coloneq \xs = \xx + Y_{\ot},
\end{align*}
the generator of the magnetic frame flow and for $2 \leq j \leq n$, we set 
\begin{align}
B_{e_j}^{\mg} = - [\xs, Y_{e_1 \wedge e_j}].
\end{align}
More generally, if $x \in \mathcal{C}^\infty(FM, \R^n)$, we denote in the same way as before 
\begin{align*}
B_x^{\mg} = \sum_{i=1}^n \langle x, e_i \rangle B_{e_i}^{\mg} \in \mathcal{C}^\infty(FM, TFM).
\end{align*}
The \emph{contorsion} tensor $\con{} \in \mathcal{C}^\infty(FM, (\R^n)^* \otimes \frk{so}(n))$ is then defined by 
\begin{align}
B_{x}^{\mg} = B_x - Y_{\con{x}}, \quad \forall x \in \R^n.
\end{align}
\end{df}

Let us now compute the contorsion tensor. We will need the following lemma:
\begin{lem} \label{lem:commutation_rel}
Let $\xi \in \frk{so}(n)$. The following formulas hold:
\begin{align}
&[Y_{\oo}, Y_\xi] = 0, \label{eq:comm1}\\
&[Y_{e_1 \wedge \oo e_1}, Y_\xi] = - Y_{\xi e_1 \wedge \oo e_1 + e_1 \wedge \oo \xi e_1}. \label{eq:comm2}
\end{align}
\end{lem}

\begin{proof}
We first prove Equation \eqref{eq:comm1}. Let $f \in \mathcal{C}^\infty(FM)$, we have
\begin{align*}
[Y_{\oo}, Y_\xi] f 
&= \left [\sum_\alpha \langle \oo, \omega_\alpha \rangle Y_{\omega_\alpha}, Y_\xi \right ] f \\
&= \sum_\alpha \langle \oo, \omega_\alpha \rangle Y_{[\omega_\alpha, \xi]} f - Y_\xi \langle \oo, \omega_\alpha \rangle \cdot Y_{\omega_\alpha} f \\
&= Y_{\sum_\alpha \langle \oo, \omega_\alpha \rangle [\omega_\alpha, \xi]} f - \sum_\alpha \langle [\oo, \xi], \omega_\alpha \rangle Y_{\omega_\alpha} f \\
&= Y_{[\oo, \xi] - [\oo, \xi]} f = 0.
\end{align*}
We now prove Equation \eqref{eq:comm2}. Similarly,
\begin{align*}
[Y_{e_1 \wedge \oo e_1}, Y_\xi] 
&= \left [ \sum_\alpha \langle e_1 \wedge \oo e_1, \omega_\alpha \rangle Y_{\omega_\alpha},  Y_\xi \right ] \\
&= \sum_\alpha \langle e_1 \wedge \oo e_1, \omega_\alpha \rangle Y_{[\omega_\alpha, \xi]} - Y_\xi \langle e_1 \wedge \oo e_1, \omega_\alpha \rangle Y_{\omega_\alpha} \\
&= Y_{[e_1 \wedge \oo e_1, \xi]} - \sum_\alpha \langle e_1 \wedge [\oo, \xi] e_1, \omega_\alpha \rangle Y_{\omega_\alpha} \\
&= Y_{[e_1 \wedge \oo e_1, \xi]} - Y_{e_1 \wedge [\oo, \xi] e_1}.
\end{align*}
Then, we may simplify 
\begin{align*}
[e_1 \wedge \oo e_1, \xi] - e_1 \wedge [\oo, \xi] e_1
&= - \xi e_1 \wedge \oo e_1 - e_1 \wedge \xi \oo e_1 - e_1 \wedge (\oo \xi - \xi \oo) e_1 \\ 
&= - \xi e_1 \wedge \oo e_1 - e_1 \wedge \oo \xi e_1.
\end{align*}
\end{proof}

We obtain the following explicit expression for the contorsion:
\begin{lem} \label{lem:contorsion}
For $x \in \R^n$, the contorsion tensor $\con{x}$ is given by
\begin{align*}
\con{x} = - \langle x, e_1 \rangle \oo^0 + \frac{1}{2}(\oo e_1 \wedge x - e_1 \wedge \oo x).
\end{align*}
In the decomposition $\frk{so}(n) = \frk{so}(n-1) \oplus e_1 \wedge \R^{n-1}$, this rewrites 
\begin{align} \label{eq:decomp_contorsion}
\con{x} = - \left [\langle x, e_1 \rangle \oo^0 + \frac{1}{2} x^\perp \wedge \oo e_1 \right ] - \frac{1}{2} e_1 \wedge (\oo x + \langle x, e_1 \rangle \oo e_1).
\end{align}
\end{lem}

\begin{proof}
We first compute $\con{e_1}$. By  the definition \ref{def:mag_std}, $B_{e_1}^{\mg} = \xs = \xx + Y_{\ot}$ and thus $\con{e_1} = - \ot$.

Assume now that $x \in e_1^\perp \subset \R^n$. We have by Definition \ref{def:mag_std} 
\begin{align*}
B_x^{\mg} 
&= - [B_{e_1}^{\mg}, Y_{e_1 \wedge x}] \\ 
&= - [B_{e_1} + \frac{1}{2} Y_{\oo} + \frac{1}{2} Y_{e_1 \wedge \oo e_1}, Y_{e_1 \wedge x}] \\ 
&= B_x - 0 + \frac{1}{2} Y_{(e_1 \wedge x) e_1 \wedge \oo e_1 + e_1 \wedge \oo (e_1 \wedge x) e_1}.
\end{align*}
Here, we used the first structural equation \eqref{eq:eq_struct} and Equations \eqref{eq:comm1}, \eqref{eq:comm2}. We may further simplify the result by noticing that, as $x \perp e_1$, we have $(e_1 \wedge x) e_1 = x$. Thus,
\begin{align*}
\con{x} = - \frac{1}{2} (x \wedge \oo e_1 + e_1 \wedge \oo x).
\end{align*}
Finally, for $x \in \R^n$, we decompose $x = \langle x, e_1 \rangle e_1 + x^\perp$, which gives 
\begin{align*}
\con{x}
&= - \langle x, e_1 \rangle \ot + \frac{1}{2} (\oo e_1 \wedge x^\perp - e_1 \wedge \oo x^\perp) \\ 
&= - \langle x, e_1 \rangle (e_1 \wedge \oo e_1 + \oo^0) + \frac{1}{2} (\oo e_1 \wedge x - \langle x, e_1 \rangle \oo e_1 \wedge e_1 - e_1 \wedge \oo x + \langle x, e_1 \rangle e_1 \wedge \oo e_1) \\ 
&= - \langle x, e_1 \rangle \oo^0 + \frac{1}{2}(\oo e_1 \wedge x - e_1 \wedge \oo x).
\end{align*}
\end{proof}

\bigskip 

In this end of section, we explain how to interpret these magnetic standard fields as a Cartan connection on the bundle $FM \to SM$.
\begin{df}
The \emph{magnetic connection} $\theta^{\mg}$ on the bundle $FM \to SM$ is the Cartan connection with model pair $G = \mathrm{Is}^+(\R^n), H = \SO(n-1)$ defined by 
\begin{align}
\theta^{\mg}(Y_\xi) = \xi, \quad \theta^{\mg}(B^{\mg}_x) = x, \quad \forall \xi + x \in \frk{so}(n) \rtimes \R^n = \frk{g}. 
\end{align}
\end{df}

\begin{lem}
The $1$-form $\theta^{\mg}$ indeed defines a Cartan connection on $FM \to SM$.
\end{lem}

\begin{proof}
It is clear that $\theta^{\mg}: TFM \to \frk{g} \times FM$ is a bundle isomorphism over $FM$. Moreover, we have by definition $\theta^{\mg}(Y_\xi) = \xi$ for $\xi \in \frk{so}(n)$. In particular, for $\xi \in \frk{so}(n-1)$ (recall that the structure group of $FM \to SM$ is $\SO(n-1)$), we see that $\theta^{\mg}$ behaves as expected on horizontal fields.

It remains to show that $\theta^{\mg}$ is $\SO(n-1)$ equivariant. Let $a \in \SO(n-1), w \in FM$, then it is clear that $\theta^{\mg}_{w \cdot a} (Y_\xi \vert_{w \cdot a}) = \Ad(a)^{-1} \xi$ for $\xi \in \frk{so}(n)$. Then, for $x \in \R^n$, using the expression derived in Lemma \ref{lem:contorsion} and the facts that $a(e_1) = e_1$ and $\oo \vert_{w \cdot a} = \Ad(a)^{-1} \oo \vert_w$, we see that the contorsion $T_x$ is $\SO(n-1)$-equivariant:
\begin{align*}
\Ad(a)^{-1} T_x = T_{a(x)}. 
\end{align*}
As a result, we obtain
\begin{align*}
\theta_{w \cdot a} (B_{x}^{\mg} \vert_{w \cdot a})
&= \theta_{w \cdot a} (B_{x} - Y_{T_x} \vert_{w \cdot a}) \\ 
&= a(x) - \Ad(a)^{-1} T_{x} \\ 
&= a(x) - T_{a(x)}.
\end{align*}
This shows that $\theta$ is $\SO(n-1)$-equivariant, hence defines a Cartan connection.
\end{proof}

\subsection{Torsion and curvature} \label{sec:mag_torsion_curv}

We now introduce notions of “torsion” and “curvature” associated with the magnetic connection. 
\begin{df}[{{\cite[Definition 3.22]{sharpe2000differential}}}]
Let $Z, Z' \in \frk{g}$. The \emph{magnetic curvature function} $\mathscr{K}^{\mg}: FM \to \Lambda^2 \frk{g}^* \otimes \frk{g}$ is defined by 
\[ \mathscr{K}^{\mg}_w (Z, Z') = [Z, Z'] - \theta_w^{\mg}([(\theta^{\mg}_w)^{-1} Z, (\theta^{\mg}_w)^{-1} Z']), \quad \forall w \in FM. \]
\end{df} 
It is easy to show (see \cite[Corollary 3.10]{sharpe2000differential}) that $\mathscr{K}^{\mg}(Z, Z')$ vanishes when $Z \in \frk{h}$ or $Z' \in \frk{h}$. Hence, we may see $\mathscr{K}$ as valued in $\Lambda^2 \frk{m}^* \otimes \frk{g}$ where $\frk{m} = \frk{h}^\perp = e_1 \wedge \R^{n-1} \oplus \R^n \subset \frk{g}$.

This function $\mathscr{K}$ is $\SO(n-1)$-equivariant, hence descends to a section of an associated bundle on $SM$, which may be identified with
\begin{align*}
(\Lambda^2 (\mathcal{N} \oplus \psm^* TM))^* \otimes (\Lambda^2 \mathcal{N} \oplus \mathcal{N} \oplus \psm^* TM).
\end{align*}
This section is (naturally) called curvature tensor associated with the Cartan connection. For our purposes, we will stick to the function $\mathscr{K}$.

In order to do computations, it is useful to introduce the following quantities.

\begin{df}
We denote $\tau^{\mg} \in \mathcal{C}^\infty(FM, \Lambda^2 \frk{m}^* \otimes \R^n)$ and $\rsfm \in \mathcal{C}^\infty(FM, \Lambda^2 \frk{m}^* \otimes \frk{so}(n))$ the projections
\begin{align}
\mathscr{K}^{\mg}(Z, Z') = \rsfm(Z, Z') + \tau^{\mg}(Z, Z') \in \frk{so}(n) \rtimes \R^n.
\end{align}
\end{df}

\begin{rmq}
Although we defined $\tau^{\mg}$ to be the projection on $\R^n$, the torsion of a Cartan connection is classically defined as the projection of $\mathscr{K}$ on $\frk{m}$. Thus, $\tau^{\mg}$ is only a component of the Cartan torsion; but this is not important for us, and the current definition is more practical for computations.
\end{rmq}

We will be particularly interested in the values taken on $\R^n \subset \frk{m}$. In this case, by expanding the bracket, we have for $x, y \in \R^n$
\begin{align*}
[B^{\mg}_x, B^{\mg}_y]
&= [B_x, B_y] - [B_x, Y_{T_{y}}] - [Y_{T_{x}}, B_y] + [Y_{T_{x}}, Y_{T_{y}}] \\ 
&= - Y_{\rfm(x, y)} + B_{T_{y} x} - Y_{B_x T_{y}} - B_{T_{x} y} + Y_{B_y T_{x}} + Y_{Y_{T_{x}} T_{y}} - Y_{Y_{T_{y}} T_{x}} + Y_{[T_{x}, T_{y}]} \\ 
&= - B^{\mg}_{T_{x} y - T_{y} x} - Y_{T_{T_{x} y - T_{y} x}} + Y_{-\rfm(x, y) - B_x T_{y} + B_y T_{x} + Y_{T_{x}} T_{y} - Y_{T_{y}} T_{x} + [T_{x}, T_{y}]}.
\end{align*}

\smallskip

So, to summarize, we will be mostly interested in the quantities
\begin{align} \label{eq:def_torsion}
\tau^{\mg}(x, y) &= \con{x} y - \con{y} x, \\ 
\label{eq:def_curv}
\rsfm(x, y) &= \rfm(x, y) + (B_x \con{y} - B_y \con{x}) - (Y_{\con{x}} \con{y} - Y_{\con{y}} \con{x}) - [\con{x}, \con{y}] + \con{\tau^{\mg}(x, y)},
\end{align}
which satisfy the structure equation 
\begin{align}
- [B^{\mg}_x, B^{\mg}_y] = B^{\mg}_{\tau^{\mg}(x, y)} + Y_{\rsfm(x, y)}. \label{eq:struct_mag}
\end{align}

The function $\tau^{\mg}$ is computed as follows:
\begin{lem} \label{lem:torsion}
Let $x, y \in \R^n$, we have 
\begin{align}
\tau^{\mg}(x, y) = - \frac{1}{2} [(x \wedge y) \oo e_1 ]^\perp + \langle \oo x, y \rangle e_1 \quad \forall x, y \in \R^n.
\end{align}
In particular, we have $\tau^{\mg}(e_1, y) = \langle \oo e_1, y \rangle e_1$. 
\end{lem}

\begin{proof}
For every $x, y \in \R^n$, we have by Lemma \ref{lem:contorsion}
\begin{align*}
\con{x} y 
&= - \langle x, e_1 \rangle \oo^0 y + \frac{1}{2}(\oo e_1 \wedge x - e_1 \wedge \oo x) y \\
&= - \frac{1}{2} \langle x, e_1 \rangle (\oo y - \langle y, e_1 \rangle \oo e_1 + \langle \oo e_1, y \rangle e_1) \\
&+ \frac{1}{2}(\langle \oo e_1, y \rangle x - \langle x, y \rangle \oo e_1 - \langle e_1, y \rangle \oo x + \langle \oo x, y \rangle e_1).
\end{align*}
Remarking that $\langle x, e_1 \rangle \langle y, e_1 \rangle, \langle e_1, x \rangle \oo y + \langle e_1, y \rangle \oo x$ and $\langle x, y \rangle$ are symmetric in $x, y$, we deduce that 
\begin{align*}
\con{x} y - \con{y} x
&= \frac{1}{2}(- \langle x, e_1 \rangle \langle \oo e_1, y \rangle + \langle y, e_1 \rangle \langle \oo e_1, x \rangle) e_1 \\
&+ \frac{1}{2}( \langle \oo x, y \rangle - \langle \oo y, x \rangle) e_1 + \frac{1}{2} ( \langle \oo e_1, y \rangle x - \langle \oo e_1, x \rangle y) \\
&= - \frac{1}{2} \langle \oo e_1, (x \wedge y) e_1 \rangle e_1 + \langle \oo x, y \rangle e_1 - \frac{1}{2} (x \wedge y) \oo e_1 \\
&= \frac{1}{2} \langle (x \wedge y) \oo e_1, e_1 \rangle e_1 + \langle \oo x, y \rangle e_1 - \frac{1}{2} (x \wedge y) \oo e_1 \\
&= - \frac{1}{2} [(x \wedge y) \oo e_1 ]^\perp + \langle \oo x, y \rangle e_1.
\end{align*}
\end{proof}

We now turn to $\rsfm$.
\begin{pps} \label{pps:curvature}
Let $x, y \in \R^n$. We have 
\begin{align*}
\rsfm(x, y) &= \rfm(x, y) + B_x \con{y} - B_y \con{x} \\ 
&+ \frac{1}{4} (\Vert \oo e_1 \Vert^2 x^\perp \wedge y^\perp + [\oo^2 x]^\perp \wedge y + x \wedge [\oo^2 y]^\perp - \oo x \wedge \oo y - 2 \langle \oo, x \wedge y \rangle \oo). 
\end{align*}
\end{pps}
\begin{proof} Note that, according to the definitions \eqref{eq:def_curv} and \eqref{eq:def_torsion}, we need to show that the expression on the second line of the above formula is equal to 
\begin{align*}
- (Y_{T_x} T_y - Y_{T_y} T_x) - [T_x, T_y] + T_{\tau^{\mg}(x, y)}.
\end{align*}
By skew-symmetry and bilinearity in $x, y$, it is enough to check the formula for $y \perp e_1$ and $x = e_1$ or $x \perp {e_1}$. We begin with the case $x = e_1$, $y \perp e_1$. By Lemma \ref{lem:contorsion}, we have the expressions 
\begin{align*}
T_{e_1} = - \frac{1}{2}(\oo + e_1 \wedge \oo e_1), \quad T_y = \frac{1}{2}(\oo e_1 \wedge y - e_1 \wedge \oo y).
\end{align*}
Next, we want to compute $[Y_{T_{e_1}}, Y_{T_y}] = Y_{Y_{T_{e_1}} T_y - Y_{T_y} T_{e_1} + [T_{e_1}, T_y]}$. Note that since $Y_{\oo} \oo = [\oo, \oo] = 0$ and $Y_{\oo}$ commutes with the $Y_\xi, \xi \in \frk{so}(n)$ (by \eqref{eq:comm1}), we have $[Y_{\oo}, Y_{T_y}] = 0$. Thus, 
\begin{align*}
[Y_{T_{e_1}}, Y_{T_y}] &= - \frac{1}{2} [Y_{e_1 \wedge \oo e_1}, Y_{T_y}] \\ 
&= \frac{1}{2} Y_{- Y_{e_1 \wedge \oo e_1} T_y + Y_{T_y} e_1 \wedge \oo e_1 - [e_1 \wedge \oo e_1, T_y]}.
\end{align*}
Using that $Y_{e_1 \wedge \oo e_1} \oo = [\oo, e_1 \wedge \oo e_1] = e_1 \wedge \oo^2 e_1$ and $y \perp e_1$, we obtain
\begin{align} \label{eq:bracket00}
Y_{e_1 \wedge \oo e_1} T_y 
&= \frac{1}{2} ((e_1 \wedge \oo^2 e_1) e_1  \wedge y - e_1 \wedge (e_1 \wedge \oo^2 e_1) y) \\ 
&= \frac{1}{2} [\oo^2 e_1]^\perp \wedge y. \notag
\end{align} 
Then,
\begin{align} \label{eq:bracket01}
Y_{T_y} e_1 \wedge \oo e_1 &= e_1 \wedge [\oo, T_y] e_1 \\ 
&= \frac{1}{2} e_1 \wedge (\oo^2 e_1 \wedge y - e_1 \wedge \oo^2 y) e_1 \notag \\ 
&= \frac{1}{2} (-\Vert \oo e_1 \Vert^2 e_1 \wedge y - e_1 \wedge [\oo^2 y]^\perp). \notag
\end{align}
Similarly,
\begin{align} \label{eq:bracket02}
[e_1 \wedge \oo e_1, T_y]
&= \frac{1}{2} ((e_1 \wedge \oo e_1) \oo e_1 \wedge y + \oo e_1 \wedge (e_1 \wedge \oo e_1) y \\ 
&\quad \quad - \oo e_1 \wedge \oo y - e_1 \wedge (e_1 \wedge \oo e_1) \oo y ) \notag \\ 
&= \frac{1}{2} (- \Vert \oo e_1 \Vert^2 e_1 \wedge y + \langle \oo e_1, y \rangle e_1 \wedge \oo e_1 - \oo e_1 \wedge \oo y - \langle e_1, \oo y \rangle e_1 \wedge \oo e_1) \notag \\ 
&= \langle \oo e_1, y \rangle e_1 \wedge \oo e_1 - \frac{1}{2} \Vert \oo e_1 \Vert^2 e_1 \wedge y - \frac{1}{2} \oo e_1 \wedge \oo y. \notag
\end{align}
Thus, combining \eqref{eq:bracket00}, \eqref{eq:bracket01} and \eqref{eq:bracket02}, we have 
\begin{align*}
Y_{T_{e_1}} T_y - Y_{T_y} T_{e_1} + [T_{e_1}, T_y] 
&= \frac{1}{4} (- [\oo^2 e_1]^\perp \wedge y - e_1 \wedge [\oo^2 y]^\perp \\ 
&\quad \quad - 2 \langle \oo e_1, y \rangle e_1 \wedge \oo e_1 + \oo e_1 \wedge \oo y).
\end{align*}
On the other hand, by Lemma \ref{lem:contorsion}, for $x = e_1$, we have
\begin{align*}
T_{\tau^{\mg}({e_1}, y)} = \langle \oo e_1, y \rangle T_{e_1} = - \frac{1}{2} \langle \oo e_1, y \rangle (\oo + e_1 \wedge \oo e_1).
\end{align*}
Thus, 
\begin{align} \label{eq:curv1}
&- (Y_{T_{e_1}} T_y - Y_{T_y} T_{e_1} + [T_{e_1}, T_y]) + T_{\tau^{\mg}(e_1, y)} \\ 
&\quad = - \frac{1}{2} \langle \oo e_1, y \rangle \oo + \frac{1}{4} ([\oo^2 e_1]^\perp \wedge y + e_1 \wedge [\oo^2 y]^\perp - \oo e_1 \wedge \oo y). \notag
\end{align}

\bigskip 

We now tackle the case when $x, y \perp e_1$. By Lemma \ref{lem:contorsion}, we have the expressions 
\begin{align*}
T_x = \frac{1}{2}(\oo e_1 \wedge x - e_1 \wedge \oo x), \quad T_y = \frac{1}{2}(\oo e_1 \wedge y - e_1 \wedge \oo y).
\end{align*}
First, we compute using \eqref{eq:diff_omega}
\begin{align*}
Y_{T_x} T_y = \frac{1}{2} ([\oo, T_x] e_1 \wedge y - e_1 \wedge [\oo, T_x] y).
\end{align*}
Using that $[\oo, T_x] = \frac{1}{2} (\oo^2 e_1 \wedge x - e_1 \wedge \oo^2 x)$, we obtain first that 
\begin{align*}
[\oo, T_x] e_1  \wedge y
&= \frac{1}{2} (-\Vert \oo e_1 \Vert^2 x - [\oo^2 x]^\perp) \wedge y.
\end{align*}
Then, 
\begin{align*}
e_1 \wedge [\oo, T_x] y 
= \frac{1}{2} e_1 \wedge ( \langle \oo^2 e_1, y \rangle x - \langle x, y \rangle \oo^2 e_1).
\end{align*}
Thus, 
\begin{align*}
Y_{T_x} T_y = \frac{1}{4} (- \Vert \oo e_1 \Vert^2 x \wedge y - [\oo^2 x]^\perp \wedge y - \langle \oo^2 e_1, y \rangle e_1 \wedge x + \langle x, y \rangle e_1 \wedge \oo^2 e_1).
\end{align*}
By symmetry, we have analogous formulas for $Y_{T_y} T_x$. We can now compute $Y_{T_x} T_y - Y_{T_y} T_x$:
\begin{align} \label{eq:bracket10}
Y_{T_x} T_y - Y_{T_y} T_x
&= \frac{1}{4} (- 2 \Vert \oo e_1 \Vert^2 x \wedge y - [\oo^2 x]^\perp \wedge y - x \wedge [\oo^2 y]^\perp + e_1 \wedge ((x \wedge y) \oo^2 e_1)).
\end{align}
We now compute $[T_x, T_y]$:
\begin{align*}
[T_x, T_y] &= \frac{1}{2} (T_x \oo e_1 \wedge y + \oo e_1 \wedge T_x y - T_x e_1 \wedge \oo y - e_1 \wedge T_x \oo y) \\ 
&= \frac{1}{4} ((\Vert \oo e_1 \Vert^2 x - \langle \oo e_1, x \rangle \oo e_1 + \langle \oo x, \oo e_1 \rangle e_1) \wedge y \\ 
&\quad \quad + \oo e_1 \wedge (\langle \oo e_1, y \rangle x + \langle \oo x, y \rangle e_1) \\ 
&\quad \quad + [\oo x]^\perp \wedge \oo y - e_1 \wedge ( \langle \oo e_1, \oo y \rangle x - \langle x, \oo y \rangle \oo e_1 - \langle \oo y, e_1 \rangle \oo x)) \\ 
&= \frac{1}{4} (-2 \langle \oo x, y \rangle e_1 \wedge \oo e_1 + \Vert \oo e_1 \Vert^2 x \wedge y + \oo x \wedge \oo y \\ 
&\quad \quad - \langle x, \oo^2 e_1 \rangle e_1 \wedge y + \langle \oo^2 e_1, y \rangle e_1 \wedge x \\ 
&\quad \quad + \langle \oo e_1, x \rangle e_1 \wedge \oo y - \langle \oo e_1, y \rangle e_1 \wedge \oo x \\ 
&\quad \quad + \langle \oo e_1, x \rangle y \wedge \oo e_1 - \langle \oo e_1, y \rangle x \wedge \oo e_1).
\end{align*}
This can be rewritten as
\begin{align} \label{eq:bracket11}
[T_x, T_y] &= \frac{1}{4} (-2 \langle \oo, x \wedge y \rangle e_1 \wedge \oo e_1 + \Vert \oo e_1 \Vert^2 x \wedge y + \oo x \wedge \oo y \\ 
&\quad \quad - e_1 \wedge ((x \wedge y) \oo^2 e_1) + e_1 \wedge (\oo (x \wedge y) \oo e_1) + (x \wedge y) \oo e_1 \wedge \oo e_1). \notag 
\end{align}
We now compute $T_{\tau^{\mg}(x, y)}$. By Lemma \ref{lem:torsion}, we have 
\begin{align*}
T_{\tau^{\mg}(x, y)} = - \frac{1}{2} T_{[(x \wedge y) \oo e_1]^\perp} + \langle \oo, x \wedge y \rangle T_{e_1}.
\end{align*}
Remark that $[(x \wedge y) \oo e_1]^\perp = (x \wedge y) \oo e_1$, as $x, y$ are orthogonal to $e_1$. Thus, 
\begin{align} \label{eq:bracket12}
T_{\tau^{\mg}(x, y)} 
= - \frac{1}{4} (\oo e_1 \wedge ((x \wedge y) \oo e_1) - e_1 \wedge (\oo (x \wedge y) \oo e_1)) - \frac{1}{2} \langle \oo, x \wedge y \rangle (\oo + e_1 \wedge \oo e_1).
\end{align}
By putting together \eqref{eq:bracket10}, \eqref{eq:bracket11} and \eqref{eq:bracket12}, we obtain 
\begin{align} \label{eq:curv2}
&- (Y_{T_x} T_y - Y_{T_y} T_x + [T_x, T_y]) + T_{\tau^{\mg}(x, y)} \\
&\quad = \frac{1}{4} (\Vert \oo e_1 \Vert^2 x \wedge y + [\oo^2 x]^\perp \wedge y + x \wedge [\oo^2 y]^\perp - \oo x \wedge \oo y - 2 \langle \oo, x \wedge y \rangle \oo). \notag
\end{align}
By combining \eqref{eq:curv1} and \eqref{eq:curv2} with the definition \eqref{eq:def_curv}, we obtain the desired result.
\end{proof}

\bigskip

We now prove a Bianchi identity. Instead of trying to use the general Bianchi identity for Cartan connections, we compute directly the cyclic sums associated with the curvature. This is both shorter and easier, since we already had to derive explicit expressions for the curvature before.

As is customary, we introduce the notation $\frk{S}$ to denote, given a tensor $A$, the cyclic sum in the first three arguments of $A$:
\begin{align*}
\frk{S} A(x, y, z, w) = A(x, y, z, w) + A(y, z, x, w) + A(z, x, y, w).
\end{align*}
We also denote $\frk{A}$ for the alternating sum on every permutation of the first three indices.

\begin{pps} \label{pps:bianchi}
For $x, y, z \in \R^n$, we have 
\begin{align}
&\frk{S} \rsfm(x, y) z = \doo(x, y, z) e_1 \\ 
&\quad + \frk{S} \left [ \frac{1}{4} \langle x \wedge z, e_1 \wedge \oo^2 e_1 \rangle y + \frac{1}{2} \langle (\xx \oo) z, x \rangle y^\perp - \frac{1}{2} \doo(x, e_1, z) y^\perp \right ]. \notag
\end{align}
Here, $\doo$ denotes the vector-valued function on $FM$ associated with the 3-form $\dmg$ on $M$.
\end{pps}

\begin{proof} Recall that by Proposition \ref{pps:curvature}, we have 
\begin{align*}
\rsfm(x, y) &= \rfm(x, y) + B_x \con{y} - B_y \con{x} \\ 
&+ \frac{1}{4} (\Vert \oo e_1 \Vert^2 x^\perp \wedge y^\perp + [\oo^2 x]^\perp \wedge y + x \wedge [\oo^2 y]^\perp - \oo x \wedge \oo y - 2 \langle \oo, x \wedge y \rangle \oo). 
\end{align*}
The computation is split into three parts. First, by the classical Riemannian Bianchi identity, we have 
\begin{align}
\frk{S} \ \rfm(x, y) z = 0.
\end{align}
\bigskip 
We now investigate the sum $\frk{S} (B_x T_y - B_y T_x)z = \frk{A} B_x T_y z$. Using Lemma \ref{lem:contorsion} we compute
\begin{align*}
B_x T_y &= - \frac{1}{2} \langle y, e_1 \rangle ((B_x \oo) - e_1 \wedge (B_x \oo) e_1) + \frac{1}{2} ((B_x \oo) e_1 \wedge y - e_1 \wedge (B_ x \oo) y) \\ 
&= -\frac{1}{2} (\langle y, e_1 \rangle B_x \oo + e_1 \wedge B_x \oo y) + \frac{1}{2} \underbrace{(B_x \oo) e_1 \wedge y^\perp}_{III}.
\end{align*}
Further, we may split 
\begin{align*}
&(\langle y, e_1 \rangle (B_x \oo) + e_1 \wedge (B_x \oo) y) z \\ 
&\quad = \underbrace{\langle y, e_1 \rangle (B_x \oo) z + \langle z, e_1 \rangle (B_x \oo) y}_{(I)} - \underbrace{\langle (B_x \oo) y, z \rangle e_1}_{(II)}.
\end{align*}
Considering that $(I)$ is symmetric with respect to $y, z$, the associated alternating sum vanishes. Then, we can immediately compute the cyclic sum associated with $(II)$ using the formula 
\[ \doo(x, y, z) = \langle (B_x \oo) y, z \rangle + \langle (B_y \oo) z, x \rangle + \langle (B_z \oo) x, y \rangle. \] 
We obtain 
\begin{align*}
\frk{A} \ \langle (B_x \oo) y, z \rangle e_1 = 2 \doo(x, y, z) e_1.
\end{align*}
Finally, to compute the alternating sum associated with $(III)$, remark first that 
\begin{align*}
(y^\perp \wedge (B_x \oo) e_1) z 
&= \langle y^\perp, z \rangle (B_x \oo) e_1 - \langle (B_x \oo) e_1, z \rangle y^\perp.
\end{align*}
Thus, the alternating sum of the first term vanishes by symmetry in the $y, z$ variables. The alternating sum of the second term gives 
\begin{align*}
& \frk{A} \ - \langle B_x \oo e_1, z \rangle y^\perp \\ 
&\quad = \frk{S} \ (- \langle B_x \oo e_1, z \rangle + \langle B_z \oo e_1, x \rangle) y^\perp \\ 
&\quad = \frk{S} \ (\langle \xx \oo z, x \rangle - \doo(x, e_1, z)) y^\perp.
\end{align*}
Hence:
\begin{align*}
\frk{S} \ (B_x T_y - B_y T_x) z = \frk{A} \ B_x T_y z 
= \doo(x, y, z) e_1 + \frk{S} \ \left [ \frac{1}{2} \langle (\xx \oo) z, x \rangle y^\perp - \frac{1}{2} \doo(x, e_1, z) y^\perp \right ].
\end{align*}

\bigskip 

We can now focus on the remaining terms, which do not depend on the horizontal derivatives of $\oo$. First, we have by straightforwardly developing the expression that
\begin{align*}
\frk{S} (x^\perp \wedge y^\perp) z = \frk{S} (x^\perp \wedge y^\perp) z^\perp = 0.
\end{align*} 
Secondly,
\begin{align*}
([\oo^2 x]^\perp \wedge y + x \wedge [\oo^2 y]^\perp)z 
= \langle [\oo^2 x]^\perp , z \rangle y - \langle y, z \rangle [\oo^2 x]^\perp + \langle x, z \rangle [\oo^2 y]^\perp - \langle [\oo^2 y]^\perp, z \rangle x.
\end{align*}
When applying $\frk{S}$, the terms in $\langle x, z \rangle, \langle y, z \rangle$ will cancel out with their permuted versions by symmetry. Thus, it remains 
\begin{align*}
\frk{S} ([\oo^2 x]^\perp \wedge y + x \wedge [\oo^2 y]^\perp)z 
&= \frk{S} ( \langle [\oo^2 x]^\perp , z \rangle y - \langle [\oo^2 y]^\perp, z \rangle x) \\ 
&= \frk{S} (\langle [\oo^2 x]^\perp, z \rangle - \langle [\oo^2 z]^\perp, x \rangle) y \\ 
&= \frk{S} (- \langle \oo^2 x, e_1 \rangle \langle e_1, z \rangle + \langle \oo^2 z, e_1 \rangle \langle e_1, x \rangle) y \\ 
&= \frk{S} \langle x \wedge z, e_1 \wedge \oo^2 e_1 \rangle y.
\end{align*}
At the third and fourth line, we used that $\oo^2$ is symmetric for the inner product on $\R^n$.

\medskip 

Thirdly, we have directly 
\begin{align*}
\frk{S} \ (\oo x \wedge \oo y) z + 2 \langle \oo, x \wedge y \rangle \oo z
&= \frk{S} \ 2 \langle \oo x, z \rangle \oo y + 2 \langle \oo x, y \rangle \oo z = 0.
\end{align*}
As a result, we have 
\begin{align*}
&\frk{S} \ \frac{1}{4} (\Vert \oo e_1 \Vert^2 x^\perp \wedge y^\perp + [\oo^2 x]^\perp \wedge y + x \wedge [\oo^2 y]^\perp - \oo x \wedge \oo y - 2 \langle \oo, x \wedge y \rangle \oo) z \\ 
&\quad = \frk{S} \ \frac{1}{4} \langle x \wedge z, e_1 \wedge \oo^2 e_1 \rangle y.
\end{align*}
By adding the three cyclic sums, we obtain the desired identity.
\end{proof}

\begin{rmq} \label{rmq:bianchi_kahler}
In the particular case where $\mg$ is (a constant multiple of) a Kähler form, we see that $\frk{S} \rsfm = 0$. As a result, the tensor $\rsfm$ enjoys the same symmetry properties as the Riemannian curvature $\rfm$. 

It is not clear for now how exactly $\rsfm$ relates to the holomorphic sectional curvature of $(g, \mg)$, in part because $\rsfm$ is an object defined over $SM$ and not $M$.
\end{rmq}

\subsection{Sectional curvature(s) and pinching} \label{sec:sec_curvature}

As a special case of Proposition \ref{pps:curvature}:
\begin{csq}
We have, for $x \perp e_1$:
\begin{align*}
\rsfm(x, e_1) e_1 = \rfm(x, e_1) e_1 + \frac{1}{2} [(\xx \oo) x]^\perp - B_x \oo e_1 + \frac{3}{4} \langle \oo e_1, x \rangle \oo e_1 - \frac{1}{4}  [\oo^2 x]^\perp.
\end{align*}
\end{csq}

\begin{proof}
This follows from Lemma \ref{lem:contorsion} and Proposition \ref{pps:curvature} evaluated at $y = e_1$, then contracted with $e_1$.
\end{proof}

Note that the above expression defines a tensor $\mathcal{M}^{\mg} \in \mathcal{C}^\infty(SM, \End(\mathcal{N}))$ by the formula 
\begin{align} \label{def:tan_mag_sec_curv}
\mathcal{M}^{\mg}_v x = \mathcal{R}(x, v) v - (\nabla_x \Omega) v + \frac{1}{2} [(\xx \Omega) x]^\perp + \frac{3}{4} \langle \Omega v, x \rangle \Omega v - \frac{1}{4} [\Omega^2 x]^\perp.
\end{align}
This tensor was introduced in \cite{assenza2024magnetic}, and we will call the associated quadratic form the \emph{tangential magnetic sectional curvature}. 

In our context, one can consider a more general notion of magnetic sectional curvature defined for $(p, v) \in SM$ and a 2-plane $\mathcal{P} = \mathrm{span}(x, y) \subset T_p M$ by 
\begin{align} \label{def:mag_sec_curv}
\mathrm{sec}^{\mg}(\mathcal{P}) \coloneq \frac{\langle \rsfm_v(x, y), y \wedge x \rangle}{\Vert x \wedge y \Vert^2}.
\end{align}

In particular, for $\mathcal{P} = \mathrm{span}(v, x)$ with $x \perp v$, we get back the original definition:
\begin{align}
\langle \mathcal{M}^{\mg}_v x, x \rangle = \mathrm{sec}^{\mg}(\mathcal{P}) \cdot \Vert x \Vert^2.
\end{align}
(Note that in the original definition, $\mathcal{M}^{\mg}$ is not renormalized as in the usual definition of sectional curvature.)

As several works highlighted the usefulness of the tensor $\mathcal{M}^{\mg}$ to tackle dynamical problems, it is important to understand how $\mathcal{M}^{\mg}$ relates to the full curvature $\rsfm$ and its sectional variant. For now, it seems unlikely that one can bound $\mathrm{sec}^{\mg}$ using only pinching on $\mathcal{M}^{\mg}$.

\bigskip 

In the rest of this section, we investigate how to bound the tensor $\rsfm$ using pinching on $\mathrm{sec}^{\mg}$. Our method adapts the original approach of \cite{bourguignon1978curvature}; we follow the presentation of \cite[Section 19.2]{lefeuvre2024microlocal}.

Denote $\mathbf{G} \in \Sym^2 (\Lambda^2 \R^n)$ the constant curvature tensor, defined by 
\begin{align}
\mathbf{G}(a, b) c = (a \wedge b) c, \quad \langle \mathbf{G}(a, b) c, d \rangle = \langle a, c \rangle \langle b, d \rangle - \langle b, c \rangle \langle a, d \rangle.
\end{align}
Recall that if $(M, g)$ has constant sectional curvature $-1$, then we have $\rfm_w = \mathbf{G}, \ \forall w \in FM$. 

\begin{df}
For $\delta  \in (0, 1]$ we say that $(M, g, \mg)$ has \emph{negative, $\delta$-pinched magnetic sectional curvature} if there exists a constant $K > 0$ such that 
\begin{align}
- K \leq \mathrm{sec}^{\mg} \leq - K \delta.
\end{align} 
\end{df}

In this case, we set $\rsfm_0 = \rsfm - K \frac{1+\delta}{2} \mathbf{G}$. One could see $\rsfm_0$ as the remainder term when one makes an asymptotic expansion of $\mathbf{R}^{\mg}$ in a neighborhood of constant curvature metrics. Our first goal is to bound $\rsfm_0$ in terms of $\delta$. Notice that for $a, b \in \R^n$ of unit norm, 
\begin{align*}
\vert \rsfm_0(a, b, b, a) \vert = K \frac{1-\delta}{2} (1 - \langle a, b \rangle^2) \leq K \frac{1-\delta}{2}.
\end{align*}
It is convenient to introduce the $4$-tensor obtained from $\rsfm_0$ by duality, which we still denote $\rsfm_0$:
\begin{align*}
\rsfm_0(a, b, c, d) = \langle \rsfm_0(a, b) c, d \rangle, \quad \forall a, b, c, d \in \R^n.
\end{align*}
We have the following symmetry relations for $\rsfm_0$:
\begin{align*}
\rsfm_0(a, b, c, d) = - \rsfm_0(b, a, c, d) = - \rsfm_0(a, b, d, c).
\end{align*}
Using polarization identities for the curvature, we obtain the following.

\begin{lem} \label{lem:decomp_curv}
Assume that $(M, g, \mg)$ has negative, $\delta$-pinched magnetic sectional curvature. Then, we have 
\begin{align}
\vert \rsfm_0(a, b, c, d) - \frac{1}{3} \frk{S} \rsfm_0(a, b, c, d) \vert \leq K \frac{2(1-\delta)}{3}.
\end{align}
\end{lem}
\begin{proof}
This is the same proof as \cite[Lemma 19.2.3]{lefeuvre2024microlocal}, with the difference that the Bianchi identity for $\rsfm_0$ does not necessarily hold, hence the remaining term involving the cyclic sum with respect to $a, b, c$.
\end{proof}

\section{Magnetic Pestov identities} \label{sec:mag_pestov}

\subsection{Universal Pestov identity} \label{sec:mag_pestov_fm}

We now prove Pestov identities using the structural equations derived previously. It will be convenient to introduce a magnetic horizontal gradient $\nhs$, which we define for $f \in \mathcal{C}^\infty(FM)$ by 
\begin{align}
\nhs f \coloneq \sum_{j=2}^n (B_{e_j}^{\mg} f) e_j \in \mathcal{C}^\infty(FM, \R^n).
\end{align}
Then, the definitions of magnetic standard fields \ref{def:mag_std}, torsion \eqref{eq:def_torsion} and curvature \eqref{eq:def_curv} rewrite
\begin{pps} Let $f \in \mathcal{C}^\infty(FM)$. We have the relations
\begin{align}
[\xs, \nv]f &= - e_1 \wedge (\nhs f), \label{eq:struct_mag1}\\ 
[\xs, \nhs] f &= - [(\rsfm)^T(\nv f) e_1]^\perp - (\xs f) \oo e_1. \label{eq:struct_mag2}
\end{align}
\end{pps}

\begin{proof}
The first equation is just another way of rewriting the definition of $B_{e_j}^{\mg} \coloneq - [\xs, Y_{e_1 \wedge e_j}]$ for $j \geq 2$. The second equation is a rewriting of the definitions of $\rsfm$ and $\tau^{\mg}$
(see \eqref{eq:def_curv}), using the expression of $\tau^{\mg}$ derived in Lemma \ref{lem:torsion}. By definition, for $j \geq 2$, we have
\begin{align*}
- [B_{e_1}^{\mg}, B_{e_j}^{\mg}] = Y_{\rsfm(e_1, e_j)} + B_{\tau^{\mg}(e_1, e_j)}^{\mg}.
\end{align*}
As a result, for $f \in \mathcal{C}^{\infty}(FM)$, decomposing $\rsfm$ and $\tau^{\mg}$ in the bases $(\omega_\alpha)_\alpha$ of $\frk{so}(n)$ and $(e_i)_i$ of $\R^n$, we have
\begin{align*}
- [\xs, \nhs] f
&= \sum_{j=2}^n Y_{\rsfm(e_1, e_j)} f \cdot e_j + B_{\tau^{\mg}(e_1, e_j)} f \cdot e_j \\ 
&= \sum_{j=2}^n \sum_{\alpha} \langle \rsfm(e_1, e_j), \omega_\alpha \rangle Y_{\omega_\alpha} f \cdot e_j + \sum_{j=2}^n \sum_{i=1}^n \langle \tau^{\mg}(e_1, e_j), e_i \rangle B_{e_i}^{\mg} f \cdot e_j \\ 
&= \sum_{j=2}^n \sum_{\alpha} \langle e_1 \wedge e_j, (\rsfm)^T \omega_\alpha \rangle Y_{\omega_\alpha} f \cdot e_j + \sum_{j=2}^n \sum_{i=1}^n \langle \oo e_1, e_j \rangle \langle e_1, e_i \rangle B_{e_i}^{\mg} f \cdot e_j.
\end{align*} 
Here, we used Lemma \ref{lem:torsion}, which  gives $\tau^{\mg}(e_1, e_j) = \langle \oo e_1, e_j \rangle e_1$. Continuing the computation, we have
\begin{align*} 
- [\xs, \nhs] f &= \sum_{j=2}^n \sum_{\alpha} \langle e_j, (\rsfm)^T (\omega_\alpha) e_1 \rangle Y_{\omega_\alpha} f \cdot e_j + \sum_{j=2}^n \langle \oo e_1, e_j \rangle \xs f \cdot e_j \\ 
&= \sum_{\alpha} Y_{\omega_\alpha} f \cdot [(\rsfm)^T (\omega_\alpha) e_1]^\perp + \xs f \cdot \oo e_1 \\ 
&= [(\rsfm)^T (\nv f) e_1]^\perp + \xs f \cdot \oo e_1.
\end{align*}

\end{proof}

\begin{rmq} \label{rmq:choice_lift}
Crucially, the only standard field appearing on the right is $\xs = B_{e_1}^{\mg}$. This is the reason we did the computations with the lift $\xs$ of $X^{\mg}$ and not another lift. More precisely, for $s \in \R$, consider the lift $\xx^{\mg, s} \coloneq \xs + s Y_{\oo^0}$ of $X^{\mg}$ to $SM$. There are associated standard magnetic fields $B_x^{\mg, s}$, defined in the same way as in Definition \ref{def:mag_std}. One can then define in the same way as before the curvature and the torsion associated with these standard fields. A similar computation as before yields that the torsion $\tau^{\mg, s}$ associated to $\xx^{\mg, s}$ satisfies 
\begin{align*}
\tau^{\mg, s}(e_1, y) e_1 = \langle \oo e_1, y \rangle e_1 - s [\oo y]^\perp, \quad \forall y \in e_1^\perp.
\end{align*}
Denote $\nhss$ the magnetic horizontal gradient associated with the standard fields $B_x^{\mg, s}$. A similar computation as in the above proof shows:
\begin{align*}
[\xx^{\mg, s}, \nhss] f = \xx^{\mg, s} f \cdot \oo e_1 + s \sum_{i=2}^n B^{\mg, s}_{e_i} f \cdot \oo e_i + \text{ vertical derivatives}, \quad \forall f \in \mathcal{C}^\infty(FM).
\end{align*}
We see that unless $s = 0$, the above structural equation involves other standard derivatives than $X^{\mg, s}$. This will be crucial in the proof of the Pestov identity, see Remark \ref{rmq:choice_lift_2}.
\end{rmq}

To prove the Pestov identities, we will need the following.
\begin{lem} \label{lem:skew_adj}
The operators $Y_{\oo}, Y_{e_1 \wedge \oo e_1}, Y_{\oo^0}$ and $B_x^{\mg}, x \in \R^n$ are formally skew-adjoint for the standard $L^2$ inner product.
\end{lem}
\begin{proof}
The proof relies on the elementary identity $\sum_{\alpha} \omega_\alpha [\oo, \omega_\alpha] = (n-2) \oo$. Let $x \in \R^n$. We will show that $B_x^{\mg}$ is skew-adjoint; along the way, we will prove that $Y_{\oo}, Y_{e_1 \wedge \oo e_1}$ are also skew-adjoint. Let $f, h \in \mathcal{C}^\infty(FM)$. Since the vector fields $B_x$ and $Y_\xi, \xi \in \frk{so}(n)$ preserve the volume, we have 
\begin{align*}
\langle B_x^{\mg} f, h \rangle 
&= \langle B_x f - \sum_\alpha \langle T_x, \omega_\alpha \rangle Y_{\omega_\alpha} f, h \rangle \\ 
&= - \langle f, B_x h \rangle + \langle f, Y_{T_x} h \rangle + \langle f, (\sum_\alpha Y_{\omega_\alpha} \langle T_x, \omega_\alpha \rangle) h \rangle, \\ 
\langle f, (B_x^{\mg})^* h \rangle 
&= - \langle f, B_x^{\mg} h \rangle + \langle f, (\sum_\alpha Y_{\omega_\alpha} \langle T_x, \omega_\alpha \rangle) h \rangle.
\end{align*}
It only remains to compute the sum.
\begin{align*}
\sum_\alpha Y_{\omega_\alpha} \langle T_x, \omega_\alpha \rangle 
&= \frac{1}{2} \sum_\alpha - \langle e_1, x \rangle \langle [\oo, \omega_\alpha], \omega_\alpha \rangle + \langle e_1, x \rangle \langle e_1 \wedge [\oo, \omega_\alpha] e_1, \omega_\alpha \rangle \\ 
&\quad \quad \quad + \langle [\oo, \omega_\alpha] e_1 \wedge x, \omega_\alpha \rangle - \langle e_1 \wedge [\oo, \omega_\alpha] x, \omega_\alpha \rangle.
\end{align*}
Observe that $\langle [\oo, \omega_\alpha], \omega_\alpha \rangle = \langle \oo, [\omega_\alpha, \omega_\alpha] \rangle = 0$. Notice that this proves that $Y_{\oo}$ is skew-adjoint. Then, 
\begin{align*}
\sum_\alpha \langle e_1 \wedge [\oo, \omega_\alpha] e_1, \omega_\alpha \rangle 
= - \sum_\alpha \langle e_1, \omega_\alpha [\oo, \omega_\alpha] e_1 \rangle = - (n-2) \langle e_1, \oo e_1 \rangle = 0.
\end{align*}
Observe that this computation says that $Y_{e_1 \wedge \oo e_1}$ is skew-adjoint. In particular, $Y_{\oo^0} = \frac{1}{2} Y_{\oo - e_1 \wedge \oo e_1}$ is skew-adjoint. Similarly, we have 
\begin{align*}
&\sum_\alpha \langle [\oo, \omega_\alpha] e_1 \wedge x, \omega_\alpha \rangle - \langle e_1 \wedge [\oo, \omega_\alpha] x, \omega_\alpha \rangle \\
&\quad = \sum_\alpha \langle x, \omega_\alpha [\oo, \omega_\alpha] e_1 \rangle + \langle e_1, \omega_\alpha [\oo, \omega_\alpha] x \rangle \\ 
&\quad = (n-2) \langle x, \oo e_1 \rangle + (n-2) \langle e_1, \oo x \rangle = 0,
\end{align*}
hence $B_x^{\mg}$ is skew-adjoint, as claimed.
\end{proof}

We also need the following lemma:
\begin{lem} \label{lem:struct_mag3}
We have the identity 
\begin{align*}
\nv^* (e_1 \wedge \nhs) - (e_1 \wedge \nhs)^* \nv 
&= (n-1) \xs + (n+3) Y_{\oo^0} - Y_{e_1 \wedge \oo e_1}. 
\end{align*}
\end{lem}

\begin{proof}
By \cite[Lemma 2.5]{cekic2025pestov}, we have 
\begin{align*}
\nv^* (e_1 \wedge \nh) - (e_1 \wedge \nh)^* \nv = (n-1) \xx. 
\end{align*}
Using Lemma \ref{lem:skew_adj}, we can thus write
\begin{align*}
\nv^* (e_1 \wedge \nhs) - (e_1 \wedge \nhs)^* \nv
&= \sum_{j=2}^n Y_{e_1 \wedge e_j}^* B_{e_j}^{\mg} - (B_{e_j}^{\mg})^* Y_{e_1 \wedge e_j} \\ 
&= - \sum_{j=2}^n [Y_{e_1 \wedge e_j}, B_{e_j}^{\mg}] \\ 
&= (n-1) \xx + \sum_{j=2}^n [Y_{e_1 \wedge e_j}, Y_{T_{e_j}}].
\end{align*}
Then, $[Y_{e_1 \wedge e_j}, Y_{T_{e_j}}] = Y_{Y_{e_1 \wedge e_j} T_{e_j} + [e_1 \wedge e_j, T_{e_j}]}$ with $T_{e_j} = \frac{1}{2} (\oo e_1 \wedge e_j - e_1 \wedge \oo e_j)$. On one hand, we have 
\begin{align*}
\sum_{j=2}^n Y_{e_1 \wedge e_j} T_{e_j} 
&= \frac{1}{2} \sum_{j=2}^n ([\oo, e_1 \wedge e_j] e_1) \wedge e_j - e_1 \wedge ([\oo, e_1 \wedge e_j] e_j) \\ 
&= \frac{1}{2} \sum_{j=2}^n ((\oo e_1 \wedge e_j + e_1 \wedge \oo e_j) e_1) \wedge e_j - e_1 \wedge ((\oo e_1 \wedge e_j + e_1 \wedge \oo e_j) e_j) \\ 
&= \frac{1}{2} \sum_{j=2}^n \oo e_j \wedge e_j - \langle \oo e_j, e_1 \rangle e_1 \wedge e_j - e_1 \wedge (\langle \oo e_1, e_j \rangle e_j - \oo e_1) \\ 
&= \frac{1}{2} \sum_{i=1}^n - e_i \wedge \oo e_i + \frac{n}{2} e_1 \wedge \oo e_1.
\end{align*}
(Notice that we sum from $1$ at the last line.) On the other hand, 
\begin{align*}
\sum_{j=2}^n [e_1 \wedge e_j, T_{e_j}] 
&= \frac{1}{2} \sum_{j=2}^n [e_1 \wedge e_j, \oo e_1 \wedge e_j - e_1 \wedge \oo e_j] \\ 
&= \frac{1}{2} \sum_{j=2}^n - \langle \oo e_1, e_j \rangle e_1 \wedge e_j - \oo e_1 \wedge e_1 - e_j \wedge \oo e_j - \langle e_1, \oo e_j \rangle e_1 \wedge e_j \\ 
&= \frac{n}{2} e_1 \wedge \oo e_1 - \frac{1}{2} \sum_{i=1}^n e_i \wedge \oo e_i.
\end{align*}
Using that $\sum_{i=1}^n e_i \wedge \oo e_i = 2 \oo$ (this can be seen in standard coordinates for $\oo$), we finally obtain
\begin{align} \label{eq:skew_adjoint_part}
\sum_{j=2}^n [Y_{e_1 \wedge e_j}, Y_{T_{e_j}}]
&= n Y_{e_1 \wedge \oo e_1} - Y_{2 \oo} \\ 
&= (n - 2) Y_{e_1 \wedge \oo e_1} - 4 Y_{\oo^0}. \notag
\end{align}
As a result, recalling that$\xs = \xx + Y_{\ot}$ and $e_1 \wedge \oo e_1 = \ot - \oo^0$, we have 
\begin{align*}
\nv^* (e_1 \wedge \nhs) - (e_1 \wedge \nhs)^* \nv 
& = (n-1) \xx + (n - 2) Y_{e_1 \wedge \oo e_1} - 4 Y_{\oo^0}\\ 
& = (n-1) \xs - (n+3) Y_{\oo^0} - Y_{e_1 \wedge \oo e_1}.
\end{align*}
\end{proof}

We can finally prove: 
\begin{thm}[Universal magnetic Pestov identity] \label{thm:univ_mag_pestov}
Let $f \in \mathcal{C}^\infty(FM)$. We have the following relation:
\begin{align*}
&\xs^* \nv^* \nv \xs f - \nv^* \xs^* \xs \nv f \\ 
&\quad = ((n-1) \xs - (n+3) Y_{\oo^0})^* \xs f + \nv^* [e_1 \wedge (\rsfm)^T(\nv f) e_1].
\end{align*}
\end{thm}

\begin{proof}
The proof is the same as \cite[Proposition 3.2]{cekic2025pestov}. We have (see \ref{def:tan_mag_sec_curv} for the definition of $\mathcal{M}^{\mg}$):
\begin{align*}
&(\xs)^* \nv^* \nv \xs - \nv^* (\xs)^* \xs \nv \\ 
&\quad = [(\xs)^*, \nv^*] \nv \xs - \nv^* (\xs)^* [\xs, \nv] \\ 
&\quad = [\nv, \xs]^* \nv \xs + \nv^* (\xs)^* (e_1 \wedge \nhs) \\ 
&\quad = (e_1 \wedge\nhs)^* \nv \xs - \nv^* \xs (e_1 \wedge \nhs) \\ 
&\quad = (-\nv^* (e_1 \wedge \nhs) + (e_1 \wedge \nhs)^* \nv) \xs + \nv^* (e_1 \wedge [\nhs, \xs]) \\ 
&\quad = - ((n-1) \xs - (n+3) Y_{\oo^0} - Y_{e_1 \wedge \oo e_1}) \xs + \nv^* (e_1 \wedge (\rsfm)^T e_1) \nv + Y_{e_1 \wedge \oo e_1}^* \xs \\ 
&\quad = ((n-1) \xs - (n+3) Y_{\oo^0})^* \xs + \nv^* (e_1 \wedge (\rsfm)^T e_1) \nv.
\end{align*}
Here, we used \eqref{eq:struct_mag1} at the third and fourth lines, Lemma \ref{lem:skew_adj} at the fourth line (recall that $\xs = B_{e_1}^{\mg}$ is thus skew-adjoint), Equation \eqref{eq:struct_mag2} and Lemma \ref{lem:struct_mag3} at the fifth line and Lemma \ref{lem:skew_adj} at the final line.
\end{proof}

As corollary, we have the following integral form of the identity:
\begin{csq} \label{thm:int_univ_mag_pestov}
Let $f \in \mathcal{C}^\infty(FM)$. We have the following relation:
\begin{align*}
&\Vert \nv \xs f \Vert^2 - \Vert \xs \nv f \Vert^2 \\ 
&\quad = (n-1) \Vert \xs f \Vert^2 - (n+3) \langle Y_{\oo^0} f, \xs f \rangle + \langle \rsfm(e_1 \wedge \nv f e_1), \nv f \rangle.
\end{align*}
\end{csq}

\begin{rmq} \label{rmq:choice_lift_2}
This is a follow-up to Remark \ref{rmq:choice_lift}. Let us see what happens if we do not choose as lift $\xs$, but $\xx^{\mg, s} = \xs + s Y_{\oo^0}$ instead. Recall that to $\xx^{\mg, s}$ we may associate standard fields $B_x^{\mg, s}$, torsion and curvature $\tau^{\mg, s}$, $\rfm^{\mg, s}$ as before. For simplicity, we shall consider the case of $f \in \mathcal{C}^\infty(FM)$ such that $\xx^{\mg, s} f = 0$, which is important in the study of ergodicity of $\xx^{\mg, s}$ (see Section \ref{sec:ergodicity}). Then, a similar computation as above yields 
\begin{align*}
0 = \Vert \xx^{\mg, s} \nv f \Vert^2 + \langle \nv f, \rfm^{\mg, s}(e_1, \nv f \cdot e_1) \rangle + s \sum_{i=2}^n \langle Y_{e_1 \wedge \oo e_i} f, B_{e_i}^{\mg, s} f \rangle.
\end{align*}
In particular, the horizontal derivatives do not simplify if $s \neq 0$. As a result, for $s \neq 0$, it is unclear whether such a formula can be used in order to investigate such $f$ using only curvature estimates, as was done in the Riemannian case. 
\end{rmq}

\subsection{Associated Pestov identities} \label{sec:mag_pestov_tensor}

We now deduce magnetic Pestov identities for associated bundles from the universal formula. We introduce the tensor $F^{E, \mg} \in \mathcal{C}^\infty(SM, \End(EM) \otimes \mathcal{N})$ which is associated with the function $\mathbf{F}^{E, \mg} \in \mathcal{C}^\infty(FM, \End(E) \otimes \R^{n-1})$ defined by 
\begin{align} \label{eq:extra_mag_curv_term}
\mathbf{F}^{E, \mg}(z) = \sum_{\omega_\alpha e_1 = 0} (\rho_* \omega_\alpha) z \otimes [(\rsfm)^T(\omega_\alpha) e_1]^\perp.
\end{align} 

\begin{thm}[Associated magnetic Pestov identity] \label{thm:tensor_mag_pestov}
Let $\rho: \SO(n) \to E$ be a unitary representation, $EM \to SM$ the associated bundle over $SM$, obtained by restricting $\rho$ to $\SO(n-1)$. Let $s$ be a section of $EM$. We have 
\begin{align*}
\Vert \nv \xs s \Vert^2 &= \Vert \xs \nv s \Vert^2 + (n-1) \Vert \xs s \Vert^2 - \langle \mathcal{M}^{\mg} \nv s, \nv s \rangle  \\ 
&\quad - \langle F^{E, \mg} s, \nv s \rangle + (n+3) \langle (\rho_* \Omega^0) s, \xs s \rangle.
\end{align*}
\end{thm}

\begin{proof}
The proof is the same as Theorem \ref{thm:tensor_pestov}. First, one proves vector-valued structural equations (see Proposition \ref{pps:vec_struct}) using the magnetic scalar structural equations \eqref{eq:struct_mag1}, \eqref{eq:struct_mag2}. Then, following the proof of Theorem \ref{thm:tensor_pestov}, we integrate this relation and use Fubini's theorem to obtain the result.

The main novelty is that we have to take into account the extra term $(n+3) Y_{\oo^0}^* \xs$ appearing in Theorem \ref{thm:univ_mag_pestov}. Since $\oo^0 e_1 = 0$, the action of $Y_{\oo^0}$ on the section $s$ is given by $- (\rho_* \Omega^0) s$, which corresponds to the additional term in the above theorem.
\end{proof}

As a particular case, when $\rho$ is the trivial representation, we obtain:
\begin{thm}[Magnetic Pestov identity on the unit tangent bundle] \label{thm:mag_pestov_sm}
Let $u \in \mathcal{C}^\infty(SM)$. We have 
\begin{align*}
\Vert \nv \xs u \Vert^2 &= \Vert \xs \nv u \Vert^2 + (n-1) \Vert \xs u \Vert^2 - \langle \mathcal{M}^{\mg} \nv u, \nv u \rangle.
\end{align*}
\end{thm}

\section{Frequency localization} \label{sec:freq_loc}

\subsection{Spherical harmonics} \label{sec:spherical_harmonics}

We now briefly recall the theory of spherical harmonics, which lets us decompose the space of functions on the sphere bundle $SM$ using the representation theory of $\SO(n)$. Our goal will then be to understand how the magnetic Pestov identity \ref{thm:mag_pestov_sm} behaves with respect to this decomposition. For simplicity, we only consider functions on $SM$ and not sections of tensor bundles. We refer the reader to \cite[Chapter 14]{lefeuvre2024microlocal} for an introduction to the classical theory.

The \emph{spherical harmonics decomposition} of $L^2(SM)$ is a certain decomposition in a dense sum of subspaces $\mathcal{H}^m$
\begin{align*}
L^2(SM) = \overline{\bigoplus_{m \in \N} \mathcal{H}^m}^{L^2}.
\end{align*}
The space $\mathcal{H}^m$ is classically described as a space of fiberwise polynomials in $v$. More precisely, there is an equivariant bijection between $\mathcal{H}^k$ and the set $\mathcal{S}^m_0$ of trace-free symmetric tensors of degree $m$, given by 
\begin{align*}
s \in \mathcal{S}^m_0 \mapsto \left ( (x, v) \in SM \mapsto s_x(v, \ldots, v) \right ) \in \mathcal{H}^m.
\end{align*}
One may also see $\mathcal{H}^m$ as an eigenspace of the vertical Laplacian $\Delta_{\mathbb{V}}^{SM} = \nv^* \nv$. The associated eigenvalue is $m(n+m-2)$, where $n = \dim M$.

Given $u \in \mathcal{C}^\infty(SM)$, we say that $u$ has \emph{finite degree} if $u \in \bigoplus_{k \leq m} \mathcal{H}^k$ for some $m \geq 0$. The \emph{degree} $\deg u$ of $u$ is then defined as the minimal $m$ such that this holds.

\medskip 

The action of the geodesic vector field $X$ on $L^2(SM)$ behaves very well with respect to this decomposition: we have 
\begin{align*}
X = X_+ + X_- : \mathcal{H}^m \to \mathcal{H}^{m+1} \oplus \mathcal{H}^{m-1}.
\end{align*}

More generally, if $\mg$ is a $2$-form on $M$, then the magnetic vector field $X^{\mg} = X + (\Omega v)^\vee$ will also behave nicely:
\begin{lem}
The operator $u \mapsto (\Omega v)^\vee u$ preserves the decomposition in spherical harmonics:
\begin{align*}
\forall u \in \mathcal{H}^m, (\Omega v)^\vee u \in \mathcal{H}^m.
\end{align*}
In particular, for $m \geq 1$, 
\begin{align*}
X = X_- + (\Omega v)^\vee + X_+ : \mathcal{H}^m \to \mathcal{H}^{m-1} \oplus \mathcal{H}^m \oplus \mathcal{H}^{m+1}.
\end{align*}
\end{lem}
\begin{proof}
Remark that for $u \in \mathcal{C}^\infty(SM)$, we have 
\begin{align*}
Y_{\oo} \pi^* u = \pi^* (\Omega v)^\vee u.
\end{align*}
Then, the operator $Y_{\oo}$ commutes with the vertical Laplacian $\Delta_{\mathbb{V}}^{FM}$ since it commutes with the fundamental vector fields by Lemma \ref{lem:commutation_rel}. It follows that $(\Omega v)^\vee$ commutes with the vertical Laplacian $\Delta_{\mathbb{V}}^{SM}$ on $SM$, and thus preserves its eigenspaces.
\end{proof}

\subsection{Localized Pestov identity} \label{sec:loc_mag_pestov}

In this subsection, we prove a magnetic Pestov identity specifically for spherical harmonics on $SM$. 

\medskip

We will need to define a magnetic horizontal gradient: given $u \in \mathcal{C}^\infty(SM)$, we set
\begin{align} 
\quad \nhs u \vert_v &= \sum_{j=2}^n B^{\mg}_{e_j} (\pi^* u) \vert_w \cdot w(e_j) \in \mathcal{C}^\infty(SM, \mathcal{N}), \quad \forall w \in FM, v = w(e_1).
\end{align} 

This gradient $\nhs$ can be expressed explicitly in terms of the Riemannian horizontal gradient $\nh$ as follows.
\begin{lem} \label{lem:mag_hor_grad}
Let $u \in \mathcal{C}^\infty(SM)$. We have the formula 
\begin{align*}
\nhs u = \nh u - \Omega^0 \nv u.
\end{align*}
Moreover, we have the identity 
\begin{align*}
\nv^* \Omega^0 \nv u &= - \frac{n-2}{2} (\Omega v)^\vee u.
\end{align*}
\end{lem}

\begin{proof}
By definition, for $2 \leq j \leq n$, we have $B^{\mg}_{e_j} = B_{e_j} + Y_{T_{e_j}}$, where $T$ is the contorsion. By Lemma \ref{lem:contorsion}, we have the expression 
\begin{align*}
T_{e_j} = \frac{1}{2} (\oo e_1 \wedge e_j - e_1 \wedge \oo e_j).
\end{align*}

Given $u \in \mathcal{C}^\infty(SM)$, the pullback $\pi^* u$ is $\SO(n-1)$ equivariant, hence the action of the vertical field writes 
\begin{align*}
Y_{T_{e_j}} \pi^* u 
= \frac{1}{2} (0 + Y_{e_1 \wedge \oo e_j}) \pi^* u = \pi^* \frac{1}{2} \langle [\Omega w(e_j)]^\perp, \nv u \rangle = - \pi^* \frac{1}{2} \langle w(e_j), \Omega \nv u \rangle.
\end{align*}
Here, we used that $\nv u$ is valued in $\mathcal{N}$. By summing and projecting on $SM$ the above, we finally obtain 
\begin{align*}
\nhs u = \nh u - \frac{1}{2} [\Omega \nv u]^\perp = \nh u - \Omega^0 \nv u.
\end{align*}

We now prove the identity. Note that the operator $P = - \nv^* \Omega^0 \nv$ is skew-adjoint, since $\Omega^0$ is skew-symmetric. By definition, this operator is the restriction to $\pi^* \mathcal{C}^\infty(SM)$ of $Q = \sum_{j=2}^n Y_{e_1 \wedge e_j}^* Y_{T_{e_j}}$. 

By Equation \eqref{eq:skew_adjoint_part} and Lemma \ref{lem:skew_adj}, we have 
\begin{align*}
(Q - Q^*) \pi^* u &= - \sum_{j=2}^n [Y_{e_1 \wedge e_j}, Y_{T_{e_j}}] \\ 
&= (n-2) Y_{e_1 \wedge \oo e_1} \pi^* u - 4 Y_{\oo^0} \pi^* u = (n-2) \pi^* (\Omega v)^\vee u.
\end{align*}
Denote $\pi_*: \mathcal{C}^\infty(FM) \to \mathcal{C}^\infty(SM)$ the adjoint of the pullback $\pi^*: \mathcal{C}^\infty(SM) \to \mathcal{C}^\infty(FM)$; this is the fiberwise integration operator. Since $\pi^* P = Q \pi^*$ and using that $\pi_* \pi^* = \identity$ (since the fibers of $FM \to SM$ have volume $1$), we have 
\begin{align*}
\pi_* Q^* \pi^* = (Q \pi^*)^* \pi^* = P^* \pi_* \pi^* = P^*.
\end{align*}
As a result, 
\begin{align*}
2P = P - P^* = \pi_* (Q - Q^*) \pi^* = (n-2) (\Omega v)^\vee.
\end{align*}
We conclude that 
\begin{align*}
- \nv^* \Omega^0 \nv u = \frac{1}{2} P u = \frac{n-2}{2} (\Omega v)^\vee u.
\end{align*}
\end{proof}

Using the above Lemma, we deduce the following decomposition of the magnetic horizontal gradient:
\begin{csq} \label{cor:decomp_hor_grad}
Let $m \geq 1, n \geq 2, m + n \geq 4$ and $u \in \mathcal{C}^\infty(SM)$. We have a decomposition in orthogonal summands 
\begin{align*}
\nhs u = \frac{1}{m+1} \nv X_+ u - \frac{1}{n+m-3} X_- u 
+ \frac{n-2}{m(m+n-2)} \nv (\Omega v)^\vee u + \mathcal{Z}^{\mg}_m(u),
\end{align*}
where $u \mapsto \mathcal{Z}^{\mg}_m u$ is a differential operator such that $\nv^* \mathcal{Z}^{\mg}_m(u) = 0$.
\end{csq}

\begin{proof}
By \cite[Lemma 14.3.3]{lefeuvre2024microlocal}, there exists a differential operator $\mathcal{Z}_m$ such that $\nv^* \mathcal{Z}_m(u) = 0$ and
\begin{align*}
\nh u  = \frac{1}{m+1} \nv X_+ u - \frac{1}{n+m-3} \nv X_- u + \mathcal{Z}_m(u).
\end{align*}
Define 
$$\mathcal{Z}_m^{\mg} u = \mathcal{Z}_m u - \frac{n-2}{2m(m+n-2)} \nv [(\Omega v)^\vee u] - \Omega^0 \cdot \nv u.$$
By Lemma \ref{lem:mag_hor_grad} and as $\nv^* \mathcal{Z}_m(u) = 0$, we also have $\nv^* \mathcal{Z}_m^{\mg}(u) = 0$. Lemma \ref{lem:mag_hor_grad} then rewrites : 
$$\nhs u  = \frac{1}{m+1} \nv X_+ u - \frac{1}{n+m-3} \nv X_- u + \frac{n-2}{2m(n+m-2)} \nv (\Omega v)^\vee u + \mathcal{Z}_m^{\mg}(u).$$
As $\nv^* \mathcal{Z}_m^{\mg}(u) = 0$ and using the orthogonality of the spaces of spherical harmonics, the above summands are orthogonal.
\end{proof}

We can finally state:

\begin{thm}[Localized magnetic Pestov identity]
Assume that $n = \dim M \geq 2$. Let $u \in \mathcal{H}^m$ be a spherical harmonic of degree $m \geq 1$, with $m + n \geq 4$. Then :
\begin{align*}
&\frac{(n+m-2)(n+2m-4)}{n+m-3} \Vert X_- u \Vert^2 - \frac{m(n+2m)}{m+1} \Vert X_+ u \Vert^2 \\
&\quad + \left [1 + \frac{(n-2)^2}{4m(n+m-2)} \right ] \Vert (\Omega v)^\vee u \Vert^2 + \Vert \mathcal{Z}_m^{\mg}(u) \Vert^2 \\
&\quad = \langle \mathcal{M}^{\mg} \nv u, \nv u \rangle.
\end{align*}
Here, $\mathcal{Z}_m^{\mg}$ is a differential operator which may be described in terms of $Z_m \coloneq Z^{\mg = 0}_m$ as
$$\mathcal{Z}_m^{\mg} u = Z_m u - \frac{n-2}{2m(m+n-2)} \nv [(\Omega v)^\vee u] - \Omega^0 \cdot \nv u.$$
Additionally, $\nv^* \mathcal{Z}_m^{\mg}(u) = 0$.
\end{thm}

\begin{proof}
The proof follows closely the proof of the localized Pestov identity without magnetic fields. 

As $\nhs = - [\xs, \nv]$ on $FM$ by definition, projecting this identity on $SM$, we obtain 
\begin{align}
\Vert \xs \nv u \Vert^2 
= \Vert \nv \xs u \Vert^2 - 2 \langle \nv \xs u, \nhs u \rangle + \Vert \nhs u \Vert^2. \label{eq:loc_pestov_0}
\end{align}
On one hand, using \cite[Lemma 14.3.2]{lefeuvre2024microlocal} and the above Lemma \ref{lem:mag_hor_grad}, we have
\begin{align*}
-2 \nv^* \nhs u &= +2(m-1) X_- u - 2(n+m-1) X_+ u + 2 \nv^* \Omega^0 \nv u \\
&= +2(m-1) X_- u - 2(n+m-1) X_+ u - (n-2) (\Omega v)^\vee u.
\end{align*} 
As a result,
\begin{align}
&-2 \langle \nv \xs u, \nhs u \rangle \notag \\
&\quad = -2 \langle \xs u, \nv^* \nhs u \rangle \notag \\
&\quad = \langle X_+ u + X_- u + (\Omega v)^\vee u, 2(m-1) X_- u - 2 (n+m-1) X_+ u - (n-2) (\Omega v)^\vee u \rangle \notag \\
&\quad = 2(m-1) \Vert X_- u \Vert^2 - 2(n+m-1) \Vert X_+ u \Vert^2 - (n-2) \Vert (\Omega v)^\vee u \Vert^2. \label{eq:loc_pestov_1}
\end{align}
Here, we used the orthogonality between the spaces of harmonics $\mathcal{H}_k$.

On the other hand, using Corollary \ref{cor:decomp_hor_grad},
\begin{align}
\Vert \nhs u \Vert^2
&= \frac{n+m-1}{m+1} \Vert X_+ u \Vert^2 + \frac{m-1}{n+m-3} \Vert X_- u \Vert^2 \label{eq:loc_pestov_2} \\
&+ \Vert \mathcal{Z}_m^{\mg} u \Vert^2 + \frac{(n-2)^2}{4m(n+m-2)} \Vert (\Omega v)^\vee u \Vert^2. \notag
\end{align}
Combining the Pestov identity \eqref{thm:mag_pestov_sm} with Equations \eqref{eq:loc_pestov_0}, \eqref{eq:loc_pestov_1} and \eqref{eq:loc_pestov_2}, we obtain 
\begin{align*}
& \Vert \nv \xs u \Vert^2 - \Vert \xs \nv u \Vert^2 \\
&\quad = (n-1) \Vert X_+ u \Vert^2 + (n-1) \Vert X_- u \Vert^2 + (n-1) \Vert ( \Omega v)^\vee u \Vert^2 - \langle \mathcal{M}^{\mg} \nv u, \nv u \rangle \\
&\quad = - 2(m-1) \Vert X_- u \Vert^2 + 2 (n+m-1) \Vert X_+ u \Vert^2 + (n-2) \Vert (\Omega v)^\vee u \Vert^2 \\
&\quad - \frac{n+m-1}{m+1} \Vert X_+ u \Vert^2 - \frac{m-1}{n+m-3} \Vert X_- u \Vert^2 - \frac{(n-2)^2}{4m(n+m-2)} \Vert (\Omega v)^\vee u \Vert^2 - \Vert \mathcal{Z}_m^{\mg} u \Vert^2.
\end{align*}
Hence :
\begin{align*}
&\langle \mathcal{M}^{\mg} \nv u, \nv u \rangle - \Vert \mathcal{Z}_m^{\mg} u \Vert^2 \\
&\quad = \frac{(n+m-2)(n+2m-4)}{n+m-3} \Vert X_- u \Vert^2 - \frac{m(n+2m)}{m+1} \Vert X_+ u \Vert^2 \\
&\quad + \left [1 + \frac{(n-2)^2}{4m(n+m-2)} \right ] \Vert (\Omega v)^\vee u \Vert^2.
\end{align*}
\end{proof}

We note the following corollary of the localized Pestov identity:
\begin{csq} \label{cor:conf_killing_tensors}
Let $u \in \mathcal{H}^m$ with $m \geq 1$ and $n \geq 3$. Assume that $\mathcal{M}^{\mg}$ is negative-definite. 

If $X_+ u = 0$, then $u = 0$.
\end{csq}

\begin{rmq}
We say that a symmetric tensor is a \emph{conformal Killing tensor} if it lies in the kernel of $X_+$, seen as acting on symmetric tensors. The above corollary thus states that if the tangential part of the sectional magnetic curvature is negative, then there are no nonzero conformal Killing tensors. 
\end{rmq}

\begin{proof}
Let $u \in \mathcal{H}^m$ be such that $X_+ u = 0$. Then, we have 
\begin{align*}
0 &= \frac{m(n+2m)}{m+1} \Vert X_+ u \Vert^2 \\
&= \frac{(n+m-2)(n+2m-4)}{n+m-3} \Vert X_- u \Vert^2 \\
&\quad + \left [1 + \frac{(n-2)^2}{4m(n+m-2)} \right ] \Vert (\Omega v)^\vee u \Vert^2 - \langle \mathcal{M}^{\mg} \nv u, \nv u \rangle + \Vert \mathcal{Z}_m^{\mg} u \Vert^2.
\end{align*} 
By assumption, all of these terms are nonnegative. It follows that every term vanishes. In particular, since $\mathcal{M}^{\mg}$ is negative-definite, we obtain $\nv u = 0$, or equivalently $u$ is constant. This contradicts the fact that $u \in \mathcal{H}^m, m \geq 1$; hence the result.
\end{proof}

\subsection{Tensor tomography} \label{sec:tensor_tomo}

We now use our theory to investigate the tensor tomography problem. Our result is the following:

\begin{thm} \label{thm:tomo}
Let $m \geq 1, n \geq 2$ with $m + n \geq 4$.  Assume that there exists $\kappa > 0$ such that 
\begin{align*}
\langle \mathcal{M}^{\mg}_v z, z \rangle + C(m, n) \langle \Omega v, z \rangle^2 \leq - \kappa \Vert z \Vert^2, \quad \forall v \in SM, \forall z \in \mathcal{C}^\infty(SM, \mathcal{N}),
\end{align*}
where $C(m, n) = \frac{m(m-1)}{2m+n-2} + (n-2) \frac{m-1}{m}$. 

Then, the following statement holds : for any $u \in \mathcal{C}^\infty(SM)$, if $\deg X^\sigma u \leq m$, then $\deg u \leq m - 1$. 
\end{thm}

To prove Theorem \ref{thm:tomo}, we use the first approach to Theorem 4.1 in \cite{guillarmou2016xray} and estimate tails of spherical harmonics decomposition. 

\begin{rmq}
We were not able to make the second approach work, even though we have the non-existence of conformal Killing tensors (see Corollary \ref{cor:conf_killing_tensors}). 

More precisely, if one were able to prove that any $u$ satisfying $\deg X^{\mg} u < \infty$ has finite degree under the assumption $\mathcal{M}^{\mg} < 0$, then by the same reasoning as in \cite{guillarmou2016xray} one can replace the constant $C(m, n)$ by $0$ in Theorem \ref{thm:tomo}.
\end{rmq}

We now prove Theorem \ref{thm:tomo}.

\begin{proof}
Denote, for $m \geq 0$ and $u \in \mathcal{C}^\infty(SM)$,
\begin{align*}
T_{\geq m} u = \sum_{k=m}^{+\infty} u_k.
\end{align*}
The theorem is a straightforward consequence of the following estimate:
\begin{align} \label{eq:estimate_tomo}
- \langle \mathcal{M}^{\mg} \nv T_{\geq m} u, \nv T_{\geq m} u \rangle \leq \Vert \nv T_{\geq m + 1} u \Vert^2 + C(m, n) \Vert (\Omega v)^\vee u_m \Vert^2.
\end{align}
Indeed, under our assumptions, this implies
\begin{align}
\kappa \Vert \nv T_{\geq m} u \Vert^2 \leq \Vert \nv T_{\geq m + 1} u \Vert^2.
\end{align}
In particular, if $\deg X^{\mg} u \leq m$, then $T_{\geq m + 1} X^{\mg} u = 0$, hence $T_{\geq m} u = 0$ or equivalently $\deg u \leq m - 1$.

\medskip

We now prove \eqref{eq:estimate_tomo}. Remark that as $X^{\mg} = X_- + (\Omega v)^\vee + X_+$, we have $T_{\geq m+1} X^{\mg} u = T_{\geq m+1} X^{\mg} T_{\geq m} u$ and $[(\Omega v)^\vee T_{\geq m} u]_m = (\Omega v)^\vee u_m$. Hence, we may assume that $u = T_{\geq m} u$ in the following.

Decomposing $X^{\mg} u$ and using Theorem \ref{thm:mag_pestov_sm}, we have 
\begin{align} \label{eq:horrible}
&\Vert \nv X^{\mg} u \Vert^2 \\
&\quad = \Vert \nv T_{\geq m+1} X^{\mg} u \Vert^2 + m(m+n-2) \Vert (X^{\mg} u)_m \Vert^2 + (m-1)(m+n-3) \Vert (X^{\mg} u)_{m-1} \Vert^2 \notag\\ 
&\quad = \Vert \xs \nv u \Vert^2 + (n-1) \Vert  X^{\mg} u \Vert^2 - \langle \mathcal{M}^{\mg} \nv u, \nv u \rangle. \notag
\end{align}
For $k \geq 1$, denote 
\begin{align}
a_-^k = \frac{1}{n+k-3}, \quad a_0^k = \frac{n-2}{2k(k+n-2)}, \quad a_+^k = \frac{1}{k+1}.
\end{align}
Then, by Corollary \ref{cor:decomp_hor_grad}, we have the decomposition 
\begin{align*}
\nhs = a_+^k \nv X_+ + a_0^k \nv (\Omega v)^\vee - a_-^k \nv X_- + \mathcal{Z}_k^{\mg}.
\end{align*}
We use this relation to estimate $\Vert \xs \nv u \Vert^2$ by looking at the terms of order $\leq m$.
\begin{align*}
\Vert \xs \nv u \Vert^2
&= \Vert \nv X^{\mg} u - \nhs u \Vert^2 \\ 
&= \Vert (1 + a_-^m) \nv X_- u_m + (1 + a_-^{m+1}) \nv X_- u_{m+1} \\ 
&\quad \quad + (1 - a_0^m) \nv (\Omega v)^\vee u_m + \nv u' + \mathcal{Z}_m^{\mg} u \Vert^2 \\ 
&= m(m+n-2) \Vert (1 + a_-^{m+1}) X_- u_{m+1} + (1 - a_0^m) (\Omega v)^\vee u_m \Vert^2 \\
&\quad \quad + (1 + a_-^m)^2 (m-1)(m+n-3) \Vert X_- u_m \Vert^2 + \Vert \nv u' \Vert^2 + \Vert \mathcal{Z}_m^{\mg} u \Vert^2.
\end{align*}
Neglecting the higher order terms and using that $(X^{\mg} u)_{m-1} = X_- u_m$ and $(X^{\mg} u)_m = X_- u_{m+1} + (\Omega v)^\vee u_m$, we obtain 
\begin{align*}
\Vert \xs \nv u \Vert^2 & \geq m(m+n-2) \Vert (1 + a_-^{m+1}) (X^{\mg} u)_{m} - (a_-^{m+1} + a_0^m) (\Omega v)^\vee u_m \Vert^2 \\ 
&\quad + (1 + a_-^m)^2 (m-1)(m+n-3) \Vert X_- u_m \Vert^2.
\end{align*}
Expanding this and replacing into \eqref{eq:horrible}, we obtain 
\begin{align*}
&\Vert \nv T_{\geq m+1} X^{\mg} u \Vert^2 + m(m+n-2) \Vert (X^{\mg} u)_m \Vert^2 + (m-1)(m+n-3) \Vert (X^{\mg} u)_{m-1} \Vert^2 \\ 
&\quad \geq m(m+n-2) (1 + a_-^{m+1})^2 \Vert (X^{\mg} u)_m \Vert^2 \\ 
&\quad \quad - 2 m(m+n-2) (1+a_-^{m+1})(a_-^{m+1} + a_0^m) \langle (X^{\mg} u)_m, (\Omega v)^\vee u_m \rangle \\ 
&\quad \quad + m(m+n-2) (a_0 + a_-^{m+1})^2 \Vert (\Omega v)^\vee u_m \Vert^2 \\ 
&\quad \quad + (1+a_-^m)^2 (m-1)(m+n-3) \Vert (X^{\mg} u)_{m-1} \Vert^2 \\ 
&\quad \quad + (n-1) \Vert X^{\mg} u \Vert^2 - \langle \mathcal{M}^{\mg} \nv u, \nv u \rangle.
\end{align*}
Note also that 
\begin{align*}
\Vert X^{\mg} u \Vert^2 \geq \Vert (X^{\mg} u)_{m-1} \Vert^2 + \Vert (X^{\mg} u)_m \Vert^2.
\end{align*}
Using that $(m-1)(m+n-3) (1+ a_-^m)^2 + (n-1 )\geq (m-1)(m+n-3)$, we may forget the $(X^{\mg} u)_{m-1}$ terms and obtain an inequality of the form
\begin{align*}
\Vert \nv T_{\geq m+1} X^{\mg} u \Vert^2 \geq \alpha \Vert P \Vert^2 + \beta \langle P, Q \rangle + \gamma \Vert Q \Vert^2 - \langle \mathcal{M}^{\mg} \nv u, \nv u \rangle,
\end{align*}
where $P = (X^{\mg} u)_m$, $Q = (\Omega v)^\vee u$ and 
\begin{align*}
\alpha &= m(m+n-2) [(1 + a_-^m)^2 -1] + n - 1 \geq 0, \\ 
\beta &= -2m(m+n-2)(1+a_-^{m+1})(a_-^{m+1} + a_0^m) \leq 0, \\ 
\gamma &= m(m+n-2) (a_0^m + a_-^{m+1})^2 \geq 0.
\end{align*}
We may then rewrite the quadratic 
\begin{align*}
\alpha \Vert P \Vert^2 + \beta \langle P, Q \rangle + \gamma \Vert Q \Vert^2 = \Vert a P - b Q \Vert^2 - C \Vert Q \Vert^2,
\end{align*}
where $C = \beta^2/4\alpha - \gamma$. It follows that 
\begin{align*}
\Vert \nv T_{\geq m+1} X^{\mg} u \Vert^2 \geq - C \Vert (\Omega v)^\vee u_m \Vert^2 - \langle \mathcal{M}^{\mg} \nv u, \nv u \rangle,
\end{align*}
and the constant $C$ is given by 
\begin{align*}
C &= m(m+n-2) \left [\frac{(1 + a_-^{m+1})^2(a_0^m + a_-^{m+1})^2}{(1 + a_-^{m+1})^2 - 1 + \frac{n-1}{m(m+n-2)}} - (a_0^m + a_-^{m+1})^2 \right ] \\ 
&= \frac{m(m-1)}{2m+n-2} + (n-2) \frac{m-1}{m},
\end{align*}
as one can compute directly or using a computer algebra system. This finishes the proof.
\end{proof}

We can now prove Theorem \ref{thm:tensor_tomo_intro}.

\begin{proof}
If $n \geq 2, m \geq 1$ and $m + n \geq 4$, then Theorem \ref{thm:tensor_tomo_intro} is just Theorem \ref{thm:tomo}. It remains to check the result for $m + n \leq 3$.

\smallskip 

For $m = 0, n \leq 3$, following \cite[Theorem 16.1.2]{lefeuvre2024microlocal} we may use directly Theorem \ref{thm:mag_pestov_sm} to conclude. Let $u \in \mathcal{C}^\infty(SM)$ be such that $X^{\mg} u$ has degree $0$. Then, $\nv X^{\mg} u = 0$, and we obtain 
\begin{align*}
0 = \Vert \xs \nv u \Vert^2 + (n-1) \Vert X^{\mg} u \Vert^2 - \langle \mathcal{M}^{\mg} \nv u, \nv u \rangle.
\end{align*}
It follows that $X^{\mg} u = 0$. However, since $\varphi_t^{\mg}$ is assumed to be Anosov, it is ergodic; it follows that $u$ is constant. 

\smallskip

Finally, for $n = 2, m = 1$, let $u \in \mathcal{C}^\infty(SM)$ be such that $X^{\mg} u$ has degree $\leq 1$. Similarly, $\Vert \nv X^{\mg} u \Vert^2 = (n-1) \Vert X^{\mg} u \Vert^2$, and we obtain 
\begin{align*}
0 = \Vert \xs \nv u \Vert^2 - \langle \mathcal{M}^{\mg} \nv u, \nv u \rangle.
\end{align*}
Under our assumption, $C(m, n) = 0$ thus $\mathcal{M}^{\mg}$ is negative definite. It follows that $\nv u = 0$, or equivalently $\deg u = 0$, as claimed.
\end{proof}

\section{Magnetic frame flow ergodicity} \label{sec:ergodicity}

The goal of this section is to prove ergodicity for the magnetic frame flow $\Phi_t^{\mg}$ under a pinching assumption on the sectional curvature. We first recall some results on partially hyperbolic dynamics which will be used in the proof.

We shall assume first that the magnetic flow $\varphi_t^{\mg}$ is Anosov: there exists a continuous decomposition 
\begin{align*}
TSM = \R X^{\mg} \oplus E_s \oplus E_u
\end{align*}
such that the differential $d\varphi_t^{\mg}$ contracts uniformly exponentially fast on $E_s$ and expands uniformly exponentially fast on $E_u$. Since the flow $\Phi_t^{\mg}$ is an isometric extension of $\varphi_t^{\mg}$, it retains some of these properties (see \cite[Chapter 12]{lefeuvre2024microlocal}). It follows from \cite[Appendix A]{assenza2025hopf} that $\varphi_t^{\mg}$ is Anosov when $\mg$ is closed and $\mathcal{M}^{\mg}$ is negative-definite. Note also that since small perturbations of Anosov flows remain Anosov, one may take $\mg$ to be a small perturbation of a closed $2$-form with negative tangential sectional curvature.

We will use the following fact:

\begin{thm}[Transitivity group and ergodicity {\cite{lef2023isometric}}] \label{thm:transitivity_group}
Associated to $\Phi_t^{\mg}$, there is a closed subgroup $K \subset \SO(n-1)$ called \emph{transitivity group} such that:
\begin{itemize}
\item The flow $\Phi_t^{\mg}$ is ergodic for the volume on $FM$ if and only if $K = \SO(n-1)$.
\item In general, there is an isometry between the set of $L^2$ functions on $FM$ which are $\Phi_t$-invariant and $L^2(\SO(n-1)/K)$.
\item This isometry restricts to a bijection between the set of smooth $\Phi_t$-invariant functions on $FM$ and the set of smooth functions on the homogeneous space $\SO(n-1)/K$.
\end{itemize}
\end{thm}

Moreover, there are topological restrictions on the group $H$, which come from the fact that $SM$ is a sphere bundle. We have the following result:

\begin{thm} \label{thm:list_groups}
If $n = \dim M$ is odd and $n \neq 7$, then $H = \SO(n-1)$.

In general, there is a list $\mathcal{G}_n$ of subgroups of $\SO(n-1)$ such that if $H \neq \SO(n-1)$ and up to a finite cover of $M$, then $H$ is conjugate to a subgroup of an element $K \in \mathcal{G}_n$. This list is given by 
\begin{itemize}
\item If $n$ is odd and $n \neq 7$, then $\mathcal{G}_n = \emptyset$.
\item If $n = 7$, then $\mathcal{G}_n = \{ \U(3) \}$.
\item If $n$ is even and $n \neq 8, 134$, then $\mathcal{G}_n = \{\SO(q) \times \SO(n-1-q), 1 \leq q \leq \min(\rho(n) - 1, \frac{n-2}{2}) \}$, where $\rho(n)$ is the $n$-th Radon-Hurwitz number.
\item If $n = 8$, then $\mathcal{G}_n = \{ \mathbf{G}_2 \} \cup \{\SO(q) \times \SO(n-1-q), 1 \leq q \leq 3 \}$.
\item If $n = 134$, then $\mathcal{G}_n = \{ \mathbf{E}_7/(\Z/2), \SO(n-2)\}$.
\end{itemize}
\end{thm}

\begin{proof}
See {\cite[Theorem 3.8]{cekic2024frame}} for the Riemannian case. Since we only use the topology of the bundle $FM \to SM$, the argument extends to any isometric extension of an Anosov flow on $SM$.
\end{proof}

\begin{rmq}
Conjecturally, for $n = 134$, $H$ is always conjugate to a subgroup of $\SO(n-2)$. This is because the proof of the previous result uses the theory of structure groups of spheres, and it is conjectured that there is no $\mathbf{E}_7$-structure on $S^{133}$.
\end{rmq}

The homogeneous space $\mathrm{Gr}(p, n-1) = \SO(n-1) / \SO(p) \times \SO(n-1-p)$ is the real Grassmannian of $p$-planes in $\R^{n-1}$. We also note the following exceptional isomorphisms:
\begin{align}
\SO(6) / \U(3) \cong \mathbb{P}^3(\C), \quad \SO(7) / \mathbf{G}_2 \cong \mathbb{P}^7(\R).
\end{align}
These can be derived from the well-known exceptional isomorphisms $\Spin(6) \cong \SU(4)$, $\Spin(7) / G_2 \cong S^7$.

The ergodicity proof relies on a Poincaré inequality which we state now. 
\begin{lem} \label{lem:poincare}
   
Let $n \geq 4$. Consider the homogeneous principal fibration $\SO(n) \to S^{n-1}$. Denote $\nabla^{\SO(n)}: \mathcal{C}^\infty(\SO(n)) \to \mathcal{C}^\infty(\SO(n), \frk{so}(n))$ the gradient of functions on $\SO(n)$ for the metric $- \frac{1}{2} \mathrm{Tr}$. For any $u \in \mathcal{C}^\infty(\SO(n))$ of zero mean value,
\begin{align}
\Vert u \Vert^2 \leq \Vert \nabla^{\SO(n)} u \cdot e_1 \Vert^2,
\end{align}
where the norms are the $L^2$ norm for the standard volume.
\end{lem}

\begin{proof}
Denote $\Delta_{\HH} = - \sum_{j=2}^n Y_{e_1 \wedge e_j}^2$, where $Y_\xi$ denotes the vector field on $\SO(n)$ associated with $\xi \in \frk{so}(n)$ on $\SO(n)$. Notice that this is a nonnegative self-adjoint operator, and by definition,
\begin{align*}
\langle \Delta_{\HH} u, u \rangle = \Vert \nabla^{\SO(n)} u \cdot e_1 \Vert^2.
\end{align*}

We may write $\Delta_{\HH} = \Delta^{\SO(n)} - \Delta^{\SO(n-1)}$, where 
\begin{align*}
\Delta^{\SO(n)} = - \sum_{\alpha} Y_{\omega_\alpha}^2, \quad \Delta^{\SO(n-1)} = - \sum_{\omega_\alpha e_1 = 0} Y_{\omega_\alpha}^2.
\end{align*}
Here, we sum over the basis elements $\omega_\alpha$ of $\frk{so}(n)$ and the subset of $\omega_\alpha \in \frk{so}(n-1)$, which spans $\frk{so}(n-1)$.

We may use the Peter-Weyl theorem \cite[Theorem 4.20]{knapp1996lie} to decompose the representation $L^2(\SO(n))$ of $\SO(n)$:
\begin{align*}
L^2(\SO(n)) = \overline{\bigoplus_\lambda V_\lambda^{\oplus \dim V_\lambda}}^{L^2},
\end{align*}
where the $V_\lambda$ are the (up to a choice of isomorphism) distinct irreducible representations of $\SO(n)$, realized by the matrix coefficients. Then, we may decompose each $V_\lambda$ under the action of $\SO(n-1)$. By Schur's lemma, the irreducible representations of $\SO(n), \SO(n-1)$ respectively are eigenspaces of $\Delta^{\SO(n)}$, $\Delta^{\SO(n-1)}$. This implies that the decomposition of $L^2(\SO(n))$ under $\SO(n-1)$ diagonalizes $\Delta_{\HH}$, and the summands are moreover spaces of smooth functions. Using some more representation theory, one can even show that the eigenvalues of $\Delta^{\SO(n)}$ are integers, see Lemma \ref{lem:optim_poincare}. In particular, nonzero eigenvalues of $\Delta_{\HH}$ are bounded below by $1$.

\medskip

Thus, to conclude, it suffices to show that the kernel of $\Delta_{\HH}$ is the set of constant functions on $\SO(n)$. Let $u$ be in the kernel of $\Delta_{\HH}$, we may assume that $u$ is smooth  by the above discussion. It follows that $\Vert \nabla^{\SO(n)} u \cdot e_1 \Vert = 0$, hence $Y_{e_1 \wedge e_j} u = 0$ for any $2 \leq j \leq n$. Then, for $2 \leq i, j \leq n$, we have 
\begin{align*}
Y_{e_i \wedge e_j} u = [Y_{e_1 \wedge e_i}, Y_{e_1 \wedge e_j}] u = 0.
\end{align*}
Thus, the full gradient $\nabla^{\SO(n)} u$ vanishes and $u$ is constant. This ends the proof.
\end{proof}

\begin{rmq}
It is easy to see that the constant $1$ is optimal by taking $f(w) = \langle w(e_2), e_3 \rangle$.
\end{rmq}

We can now state our result.

\begin{thm} \label{thm:ergo_mag_frame_flow}
Let $(M, g)$ be a Riemannian manifold of dimension $n \geq 3$ and $\delta \in (0, 1)$. Assume that 
\begin{align*}
\delta > \delta^*(n) \coloneq \max_{K \in \mathcal{G}_n} \frac{2 \sqrt{\nu(K)}}{3+2 \sqrt{\nu(K)}},
\end{align*}
where $\nu(K)$ is the smallest nonzero eigenvalue of the Laplacian on $\SO(n-1)/K$ (for the quotient metric induced by $- \frac{1}{2} \mathrm{Tr}$ on $\SO(n)$). 

Then, there exists an explicit constant $\varepsilon = \varepsilon(\delta, g)$ such that for any $2$-form $\mg$ satisfying:
\begin{itemize}
\item[(0)] The magnetic geodesic flow $\varphi_t^{\mg}$ is Anosov,
\item[(i)] The full magnetic sectional curvature $\mathrm{sec}^{\mg}$ is $\delta$-pinched: there exists $K = K(g, \mg) > 0$ such that
\begin{align*}
-K \leq \mathrm{sec}^{\mg} \leq - K \delta,
\end{align*}
\item[(ii)] We have a bound on the tensor $\frk{S} \rsfm$ computed in Proposition \ref{pps:bianchi}:
\begin{align} \label{eq:bound_sr0}
\vert  \langle (\frk{S} \rsfm)(e_1, \xi e_1), \xi - e_1 \wedge \xi e_1 \rangle \leq K \varepsilon \Vert e_1 \wedge \xi e_1 \Vert \cdot \Vert \xi - e_1 \wedge \xi e_1 \Vert,
\end{align}
\end{itemize}
then the associated magnetic frame flow $\Phi_t^{\mg}$ is ergodic.
\end{thm}

\begin{rmq}
We explicitly compute the values taken by $\nu(K)$ in Appendix \ref{sec:casimir}; we have 
\begin{align*} 
\max_{K \in \mathcal{G}_n} \nu(K) = 
\begin{cases*}
   0 \text{ if $n \neq 7$ is odd,} \\
   16 \text{ if $n = 7, 8$,} \\
   n-2 \text{ if $n \neq 8, 134$ is even,} \\ 
   \max(132, \lambda_1(\SO(133) / (\mathbf{E}_7 / \Z_2))) \text{ if $n = 134$}.
\end{cases*}
\end{align*}
\end{rmq}

\begin{rmq}
By Remark \ref{rmq:bianchi_kahler}, for Kähler magnetic flows, $\frk{S} \rsfm = 0$ and thus condition (ii) is automatically satisfied. In particular, Theorem \ref{thm:ergo_kahler_frame_flow_intro} follows immediately from Theorem \ref{thm:ergo_mag_frame_flow}.

However, in general, the quantity $\langle \frk{S} \rsfm (e_1, \xi e_1), \xi \rangle$ depends explicitly on $\nabla \Omega$, which prevents further simplifications.
\end{rmq}

We now use this result to prove Theorem \ref{thm:ergo_mag_frame_flow_intro}.

\begin{proof} Let $(M, g)$ be a Riemannian, negatively curved manifold. Let us check that if the metric $g$ is $\delta$-pinched, $\delta > \delta^*(n)$, $\mg$ is closed and sufficiently small in $\mathcal{C}^2$-topology, then, conditions $(0), (i), (ii)$ of Theorem \ref{thm:ergo_mag_frame_flow} are satisfied.

As mentionned before, if $\mg$ is closed, then $\varphi_t^{\mg}$ is Anosov.

Then, if $\mg$ is sufficiently small in $\mathcal{C}^1$ topology, since the difference $\rsfm - \rfm$ only depends on the values taken by $\mg$ and $\nabla^{LC} \mg$, it follows that the difference between sectional curvatures $\mathrm{sec}^{\mg} - \mathrm{sec}$ and the Bianchi remainder term $\frk{S} \rsfm$ may be bounded uniformly using the $\mathcal{C}^1$ norm of $\mg$. Moreover, one can easily obtain explicit bounds using the computations we did in Propositions \ref{pps:curvature}, \ref{pps:bianchi}. 

In particular, we may ask that $\mathrm{sec}^{\mg}$ is $\delta'$-pinched for some $\delta' > \delta^*(n)$ and then bound $\frk{S} \rsfm$ in terms of $\delta'$, so that conditions $(i)$ and $(ii)$ are satisfied. This ends the proof.
\end{proof}

\medskip 

We now prove Theorem \ref{thm:ergo_mag_frame_flow}.

\begin{proof}
Assume by contradiction that $\Phi_t^{\mg}$ is not ergodic. Then, the associated group $K$ given by Theorem \ref{thm:transitivity_group} is a strict subgroup of $\SO(n-1)$, and is contained in one of the subgroups $K^{\mathrm{max}} \in \mathcal{G}_n$ listed in Theorem \ref{thm:list_groups}. In particular, we may assume that $n \geq 4$, as $\mathcal{G}_3 = \emptyset$. 

\smallskip

Let $\nu$ be the smallest nonzero eigenvalue of the Laplacian on the homogeneous space $\SO(n-1)/K^{\mathrm{max}}$ and $h$ a nonzero eigenvector for $\nu$; in particular, $h$ is not constant. Remark that $h$ lifts to a function $\tilde{h}$ on $\SO(n-1)/K$, which is also an eigenvalue of the Laplacian of $\SO(n-1)/K$ with the same eigenvalue, since the map $\SO(n-1)/K \to \SO(n-1)/K^{\mathrm{max}}$ is a Riemannian submersion. 

By Theorem \ref{thm:transitivity_group}, $\tilde{h}$ induces a smooth $\Phi_t^{\mg}$-invariant $f$ on $FM$ which is moreover orthogonal to the constant functions as the map $L^2(\SO(n-1)/K) \hookrightarrow \ker \xs \cap L^2(FM)$ is an isometry for the $L^2$ inner product.

\medskip 

We apply the universal Pestov identity to $f$. Since $\xs f = 0$, we obtain 
\begin{align*}
\Vert \xs \nv f \Vert^2 + \langle \rsfm(e_1, \nv f \cdot e_1), \nv f \rangle = 0.
\end{align*}
Decomposing $\nv f = e_1 \wedge \nv f e_1 + \nv^\perp f$, we obtain 
\begin{align*}
\Vert \xs \nv f \Vert^2 + \rsfm(e_1, \nv f \cdot e_1, e_1, \nv f \cdot e_1) + \langle \rsfm(e_1, \nv f \cdot e_1), \nv^\perp f \rangle = 0.
\end{align*}
Using the pinching, we may then bound 
\begin{align*}
\rsfm(e_1, \nv f \cdot e_1, e_1, \nv f \cdot e_1) \geq \delta \Vert \nv f e_1 \Vert^2.
\end{align*}
On the other hand, decomposing $\rsfm = K \frac{1+\delta}{2} \mathbf{G} + (\rsfm_0 - \frac{1}{3} \frk{S} \rsfm_0) + \frac{1}{3} \frk{S} \rsfm_0$ (see Section \ref{sec:sec_curvature}) and using Cauchy-Schwartz's inequality together with Lemma \ref{lem:decomp_curv} and our assumption \eqref{eq:bound_sr0}, we obtain
\begin{align*}
\vert \langle \rsfm(e_1, \nv f \cdot e_1), \nv^\perp f \rangle \vert 
&= \vert \langle \rsfm_0(e_1, \nv f \cdot e_1), \nv^\perp f \rangle \vert \\ 
&\leq K \left ( \frac{2(1-\delta)}{3} + \frac{\varepsilon}{3} \right ) \Vert \nv f \cdot e_1 \Vert \cdot \Vert \nv^\perp f \Vert.
\end{align*}
We now do a double Poincaré inequality on the $\nv^\perp f$ term. To do this, it is simpler to look first at the smooth function $f$ on a single fiber of $F_p M$ of $FM \to M$, which is isometric to $\SO(n)$. We will then use Fubini's theorem to integrate the fiberwise inequality over $M$. Under an arbitrary choice of base point, which yields an identification of $F_p M$  with $\SO(n)$, $\nv^\perp f$ corresponds to $\nabla^{\SO(n-1)} f$ (through the inclusion $\SO(n-1) \hookrightarrow \SO(n)$) and $\nv f \cdot e_1$ corresponds to $\nabla^{\SO(n)} f \cdot e_1$. 

We claim that for $p \in M$, we have $\Vert \nv^\perp f \Vert^2_{L^2(F_p M)} = \nu \Vert f \Vert^2_{L^2(F_pM)}$. The structure equation $[\xs, \nv] = - \nhs$ implies in particular $[\xs, \nv^\perp] = 0$. Denote, for $v \in SM$,
\begin{align*}
h(v) \coloneq \int_{\pi^{-1}(v)} \Vert \nv^\perp f \vert_w \Vert^2 dw.
\end{align*}
Here, we integrate with respect to a renormalized Haar measure on a single fiber of $FM \to SM$. Then, the function $h: SM \to \R$ is clearly smooth and is moreover $\varphi_t^{\mg}$-invariant:
\begin{align*}
X^{\mg} h(v) &= \int_{\pi^{-1}(v)} 2 \langle \xs \nv^\perp f, \nv^\perp f \rangle \vert_w dw \\ 
&= \int_{\pi^{-1}(v)} 2 \langle \nv^\perp \xs f, \nv^\perp f \rangle \vert_w dw \\ 
&= 0.
\end{align*}
However, by assumption, $\varphi_t^{\mg}$ is Anosov, thus it is ergodic. It follows that $h$ is constant on $SM$. Moreover, by construction, we have $h(v) = \nu \Vert f \Vert^2_{L^2(\pi^{-1} v)}$ for some $v \in SM$. By Fubini's theorem, we finally obtain 
\begin{align*}
\Vert \nv^\perp f \Vert^2_{L^2(F_p M)} = \int_{S_p M} h(v) dv = \nu \Vert f \Vert_{L^2(F_p M)}.
\end{align*}

\medskip

Using Lemma \ref{lem:poincare}, we get
\begin{align*} 
\int_{F_p M} \Vert \nabla^{\SO(n-1)} f \Vert^2
= \nu \int_{F_p M} \Vert f \Vert^2
\leq \nu \int_{F_p M} \Vert \nabla^{\SO(n)} f \cdot e_1 \Vert^2. 
\end{align*}
Integrating this inequality on $M$, we obtain by Fubini's theorem
\begin{align*}
\Vert \nv^\perp f \Vert^2 \leq \nu \Vert \nv f \cdot e_1 \Vert^2.
\end{align*}
Thus, we have the inequality
\begin{align*}
\Vert \xs \nv f \Vert^2 + K \left ( \delta - \frac{2(1-\delta)}{3}\sqrt{\nu} - \frac{\varepsilon}{3} \right ) \Vert \nv f e_1 \Vert^2 \leq 0.
\end{align*}
In particular, as $\delta > \delta^*$ and for $\varepsilon < 3 \delta - 2(1-\delta)\sqrt{\nu}$ (which is positive precisely because $\delta > \delta^*$), we have $\xs \nv f = 0$ and $\nv f e_1 = 0$, so in particular $Y_{e_1 \wedge e_j} f = 0$ for every $2 \leq j \leq n$. But then every vertical derivative of $f$ vanishes, since for $Y_{e_i \wedge e_j}$ with $2 \leq i, j \leq n$, we have 
\begin{align*}
Y_{e_i \wedge e_j} f = [Y_{e_1 \wedge e_i}, Y_{e_1 \wedge e_j}] f = 0.
\end{align*}
Moreover, $0 = \xs \nv f = - \nhs f$ as $\xs f = 0$. Thus, every standard and vertical derivative of $f$ vanishes, and $f$ is constant. This contradicts the fact that $f$ is nonconstant, and thus $\Phi^{\mg}_t$ is ergodic.
\end{proof}

\appendix

\section{Magnetic Jacobi fields} \label{sec:mag_jacobi}

The goal of this Appendix is to use our magnetic connection to define and study the magnetic Jacobi equation and magnetic Jacobi fields. Such objects were partially constructed in \cite{assenza2025hopf} and used successfully; our goal is to give a complete derivation.

\smallskip

We first write a Jacobi equation at the level of $FM$. Given $Z \in TFM$, we decompose the evolution of $Z$ along the magnetic frame flow $d\Phi_t^{\mg}$ in the decomposition given by the “standard magnetic” and vertical vector fields:
\begin{align*}
d\Phi_t^{\mg}(Z) = a(t) \xs + B^{\mg}_{H(t)} + Y_{V(t)} \in \R \xs \oplus \langle B_{e_j}^{\mg}, 2 \leq j \leq n \rangle \oplus \langle Y_\xi, \xi \in \frk{so}(n) \rangle.
\end{align*}
\begin{lem}[Magnetic Jacobi equation on the frame bundle] \label{lem:mag_jacobi}
The functions $a: \R \to \R$, $H: \R \to \R^{n-1}$ and $V: \R \to \frk{so}(n)$ satisfy the ODE
\begin{align}
\begin{cases}
a'(t) = \langle H(t), \oo e_1 \rangle, \\
H'(t) = V(t) e_1, \\
V'(t) = - \rsfm(H(t), e_1).
\end{cases}
\end{align} 
\end{lem}

\begin{proof}
We write
\begin{align*}
Z &= a(t) d\Phi_{-t}^{\mg} \xs + d\Phi_{-t}^{\mg} B^{\mg}_{H(t)} + d\Phi_{-t}^{\mg} Y_{V(t)} \\
&= a(t) \xs + \sum_{j=2}^n \langle H(t), e_j \rangle d\Phi_{-t}^{\mg} B_{e_j} + \sum_\alpha \langle V(t), \omega_\alpha \rangle d\Phi_{-t}^{\mg} Y_{\omega_\alpha}.
\end{align*}
We now differentiate this relation using the structural equations \eqref{eq:struct_mag1}, \eqref{eq:struct_mag2}, which may be written as
\begin{align*}
[\xs, Y_{\omega_\alpha}] &= - B_{\omega_\alpha e_1}^{\mg}, \\
[\xs, B_{e_j}^{\mg}] &= - Y_{\rsfm(e_1, e_j)} - B^{\mg}_{\tau^{\mg}(e_1, e_j)} = - Y_{\rsfm(e_1, e_j)} - \langle \oo e_1, e_j \rangle \xs.
\end{align*}
We obtain:
\begin{align*}
0 &= a'(t) \xs + B_{H'(t)} - Y_{\rsfm(e_1, H(t))} - \langle \oo e_1, H(t) \rangle \xs + Y_{V'(t)} - B^{\mg}_{V(t)e_1}.
\end{align*}
The result follows.
\end{proof} 

\medskip

We now project this equation on $SM$. To this end, we define a magnetic total horizontal distribution $\mathbb{H}^{\mg, SM, tot} \subset TM$ by projecting on $SM$ the distribution spanned by the standard magnetic fields $B^{\mg}_x$:
\begin{align*}
\mathbb{H}^{\mg, SM, tot}_v = \langle \pi_* B_x^{\mg} \vert_w, x \in \R^n, w \in FM, v = w(e_1) \rangle \subset TSM.
\end{align*}
Equivalently, in terms of horizontal and vertical lifts for Levi-Civita, we have the formula
\begin{align*}
\HH^{\mg, SM, tot}_v = \{w(x)^H - (w[T_x e_1]^\perp)^\vee, x \in \R^n\}, \quad w \in FM, v = w(e_1).
\end{align*}
This distribution splits as $\HH^{\mg, SM, tot}_v = \R X^{\mg} \oplus \HH^{\mg, SM}$, where $\HH^{\mg, SM}$ corresponds to the projections of $B_{x}^{\mg}, x \perp e_1$. We also denote, for $Z \in \mathcal{N}_v$, $\tilde{Z}$ the corresponding horizontal lift (and we still denote $Z^\vee$ the vertical lift).

\begin{rmq}
The distribution $\HH^{\mg, tot}$ coincides with the kernel of the twisted connector map introduced in \cite{assenza2025hopf}, when restricted to $SM$.
\end{rmq}

Using the same method, we may obtain a magnetic Jacobi equation on $SM$ for tensors using the magnetic covariant derivative along magnetic geodesics $\nabla^{\mg}_t$ (Definition \ref{def:cov_diff}):
\begin{lem}
Let $\gamma$ be a magnetic geodesic on $M$, $v = \gamma'(0)$, $Z \in T_v SM$. Then, the evolution 
\begin{align*}
d\varphi_t^{\mg}(Z) = a(t) X^{\mg} + \widetilde{H(t)} + V(t)^\vee,
\end{align*}
where $H(t)$ and $V(t)$ are sections of $\mathcal{N}$ along $\gamma'$, is given by the following ODE:
\begin{align} \label{eq:mag_jacobi_sm}
\begin{cases}
a'(t) = \langle H(t), \Omega v \rangle, \\
\nabla_t^{\mg} H(t) = V(t), \\
\nabla_t^{\mg} V(t) = - \mathcal{M}^{\mg} H(t).
\end{cases}
\end{align} 
\end{lem}

\begin{proof}
Following the proof of Lemma \ref{lem:mag_jacobi}, we have the equation
\begin{align*}
0 = a'(t) X^{\mg} + [X^{\mg}, \widetilde{H(t)}] + [X^{\mg}, V(t)^\vee].
\end{align*}

Let $\mathbf{H}: FM \to \R^{n-1}$ the $\SO(n-1)$-equivariant function associated with $H$. Similarly, let $\mathbf{V}: FM \to e_1 \wedge \R^{n-1} \subset \frk{so}(n-1)$ be the equivariant function associated with $V$. Note that the tensors $\nabla_t^{\mg} H, \nabla_t^{\mg} V$ correspond by definition to the functions $\xs \mathbf{H}, \xs \mathbf{V}$.

Then, we have 
\begin{align*}
[\xs, Y_{\mathbf{V}}] &= Y_{\xs \mathbf{V}} - B^{\mg}_{\mathbf{V} e_1}, \\ 
[\xs, B_{\mathbf{H}}^{\mg}] &= B_{\xs \mathbf{H}} - Y_{\rsfm(e_1, \mathbf{H})} - B^{\mg}_{\tau^{\mg}(e_1, \mathbf{H})}.
\end{align*}
Applying the projection $\pi: FM \to SM$, we obtain 
\begin{align*}
[X^{\mg}, V^\vee] &= (\nabla_t^{\mg} V)^\vee - \widetilde{V}, \\ 
[X^{\mg}, \widetilde{H}] &= \widetilde{\nabla_t^{\mg} H} + (\mathcal{M}^{\mg} H)^\vee - \langle \Omega v, H \rangle X^{\mg}.
\end{align*}
The result follows.
\end{proof}

We may generalize a number of definitions from classical Riemannian geometry to the magnetic context. 
\begin{df}
Let $\gamma$ be a magnetic geodesic defined on an interval $I \subset \R$. An \emph{orthogonal magnetic Jacobi field} along $\gamma$ is a section $J$ of $\mathcal{N}$ along $\gamma' \subset SM$ which satisfies the equation 
\begin{align*}
(\nabla_{t}^{\mg})^2 J(t) + \mathcal{M}^{\mg} J(t) = 0.
\end{align*}
For $s, t \in I$, we say that $\gamma(s)$, $\gamma(t)$ are \emph{conjugate along $\gamma$} if 
\begin{align*}
d\varphi_{t-s}^{\mg} (\mathbb{V}^{SM}_{\gamma(s)}) \cap \mathbb{V}^{SM}_{\gamma(t)} \neq \{0\}.
\end{align*}
\end{df}

Note that if $J$ is an orthogonal magnetic Jacobi field along a magnetic geodesic $\gamma$, then $J$ corresponds to the evolution of the tangent vector $Z = \widetilde{J(0)} + (\nabla_t^{\mg} J(0))^\vee \in T_{\gamma(0)} SM$ under $d\varphi_t^{\mg}$ through the formula
\begin{align*}
d\varphi_t^{\mg}(Z) = \left ( \int_0^t \langle J(s), \Omega \gamma'(s) \rangle ds \right ) X^{\mg} + \widetilde{J(t)} + (\nabla_t^{\mg} J(t))^\vee.
\end{align*}
In particular, our Jacobi fields project to the Jacobi fields defined in \cite{assenza2025hopf}. 

\begin{rmq}
In \cite{assenza2025hopf}, one defines conjugate points by the condition 
\begin{align*}
d\varphi_{t-s}^{\mg} (\R X^{\mg}_{\gamma(s)} \oplus \mathbb{V}^{SM}_{\gamma(s)}) \cap (\R X^{\mg}_{\gamma(t)} \oplus \mathbb{V}^{SM}_{\gamma(t)}) \neq \R X^{\mg}_{\gamma(t)}.
\end{align*}
Thus, our definition is more restrictive than theirs. For now, it is unclear whether one is better than the other.
\end{rmq}

\section{Eigenvalues of Casimir operators} \label{sec:casimir}

In this Appendix, we explain how to compute eigenvalues for Casimir operators in order to obtain an explicit expression for $\delta^*(n)$ in Theorem \ref{thm:ergo_mag_frame_flow}. We refer the reader to \cite{sepanski2007compact, fulton2013representation} for an introduction to representation theory and \cite{knapp1996lie} for more advanced topics.

\medskip 

Let $n \geq 3$. Recall that the special orthogonal group $\SO(n)$ acts naturally by conjugacy on its Lie algebra $\frk{so}(n)$, also called the \emph{adjoint representation}. There are two natural $\SO(n)$-invariant metrics that we may consider on $\frk{so}(n)$. 

First, one may see $\frk{so}(n)$ as the Lie algebra of skew-symmetric endomorphisms of $\R^n$ for the standard inner product $g_{\mathrm{std}}$. There is a natural basis of this Lie algebra given by the $e_i \wedge e_j, i < j$ where $(e_i)$ is the standard basis; we define a metric $\langle \cdot, \cdot \rangle$ by taking $(e_i \wedge e_j)_{i < j}$ to be orthonormal. This is the metric we use in the body of this article.

Second, we may consider the \emph{Killing form} $B$ defined by 
\begin{align*}
B(\xi, \eta) = - \mathrm{Tr}(\ad_\xi \circ \ad_\eta).
\end{align*}
A quick computation shows that $B = (n-2) \langle \cdot, \cdot \rangle$ \cite[p. 272]{fulton2013representation}.

The Killing form is the natural metric to consider when discussing the representation theory of $\SO(n)$ and $\frk{so}(n)$. Thus, in this Appendix, we shall primarily use $B$, and only then state results about $B$ in terms of $\langle \cdot, \cdot \rangle$.

\medskip 

Assume that $n \geq 3$, so that $\SO(n)$ is semi-simple. Let $\rho: \SO(n) \to \End(V)$ be an irreducible, unitary representation where $(V, h)$ is a Hermitian, finite dimensional vector space. Fix $T \subset \SO(n)$ a maximal torus and $\frk{t} \subset \frk{so}(n)$ the Lie algebra of $T$. The elements of $T$ act through $\rho$ on $V$, and thus the elements of $\frk{t}$ act through $\rho_* \coloneq d\rho_{\identity}$ on $V$. Since $T$ is abelian, this action may be co-diagonalized; we have a decomposition
\begin{align*}
V = \bigoplus_{\alpha \in \Delta(V)} V_\alpha, \quad V_\alpha = \{ x \in V, \forall \xi \in \frk{t}, (\rho_* \xi) x = \alpha(\xi) x \}.
\end{align*}
More precisely, on each term $V_\alpha$ of the decomposition, elements $\xi \in \frk{so}(n)$ act through a linear form $\alpha$ called a \emph{weight}.

In the particular case $V = \frk{so}(n)_\C$ of the complexified adjoint representation, nonzero weights are called \emph{roots} of the Lie algebra $\frk{so}(n)$. 

\medskip 

Weights generate a lattice of $\frk{t}$, which may be computed explicitly in terms of the parity of $n$. More precisely, for a certain orthonormal basis $(e_i)$ of $\frk{t}$ (for the Killing metric), we may define a natural basis $(\varpi_i)_i$ of the weight lattice (called \emph{fundamental weights}) which is given as follows. For $n = 2m \geq 4$, we have 
\begin{equation*} \varpi_i = \begin{cases*}
e_1 + \ldots + e_i, & $1 \leq i \leq m-2$, \\
\frac{1}{2} (e_1 + \ldots + e_{m-1} - e_m), & $i = m - 1$, \\ 
\frac{1}{2} (e_1 + \ldots + e_{m-1} + e_m), & $i = m$.
\end{cases*} 
\end{equation*}
For $n = 2m + 1 \geq 3$, we have 
\begin{align*}
\varpi_i = \begin{cases*}
e_1 + \ldots + e_i, & $1 \leq i \leq m - 1$, \\ 
\frac{1}{2} (e_1 + \ldots + e_{m-1} + e_m), & $i = m$.
\end{cases*}
\end{align*}

An element $\alpha$ of the weight lattice is called \emph{analytically integral} if, for every $\xi \in \frk{so}(n)$ such that $\exp(\xi) = \identity$, we have $\alpha(\xi) \in 2 \pi \Z$. Analytically integral weights correspond to weights coming specifically from the representations of $\SO(n)$ and not just representations of its Lie algebra, or, equivalently, its universal cover $\Spin(n)$ \cite[Theorem 5.110]{knapp1996lie}.

The analytically integral elements form a sublattice of the weight lattice. A basis of this sublattice, in terms of the fundamental weights, is the following \cite[Proposition 3.1.19]{goodman2009representations}: 
\begin{itemize}
\item For $n = 2m \geq 4$, the elements $\varpi_i, 1 \leq i \leq m-2$ and $2 \varpi_{m-1}, \varpi_{m-1} + \varpi_m, 2 \varpi_m$. 
\item For $n = 2m+1 \geq 3$, the elements $\varpi_i, 1 \leq i \leq m-1$ and $2 \varpi_m$.
\end{itemize}
In particular, we see that the analytically integral weights of representations of $\SO(n)$ have integer coefficients.

\medskip

Weights also allow one to describe the set of irreducible representations of $\frk{so}(n)$ as follows. Given an irreducible representation $\rho$, there is a unique \emph{highest weight} $\lambda$ attached to $\rho$, which may be characterized in several ways \cite[Section 5.2]{knapp1996lie}. In particular, highest weights may be described as linear combinations of the fundamental weights with nonnegative integer coefficients.

\medskip 

We now turn to Casimir operators. Fix $(\xi_\alpha)$ an orthonormal basis of $\frk{so}(n)$ for the Killing form. For instance, one may take $\xi_\alpha = (n-2)^{-1/2} \omega_\alpha$, where $(\omega_\alpha)_\alpha$ is the standard basis introduced in Section \ref{sec:levi_civita_fm}. The \emph{Casimir operator} $C_n^\rho$ associated with a unitary representation $\rho$ is defined by 
\begin{align*}
C_n^\rho = - \sum_\alpha (\rho_* \xi_\alpha)^* (\rho_* \xi_\alpha),
\end{align*}
where $\cdot^*$ denotes the adjoint for the Hermitian metric $h$. One shows that $C^n_\rho$ does not depend on the choice of basis and commutes with the action of $\frk{so}(n)$ \cite[Proposition 5.24]{knapp1996lie}.

Then, by Schur's lemma, it follows that $C_n^\rho = c_n^\rho \identity_V$ acts by scalar multiplication on $V$, and the scalar is given in terms of the highest weight $\lambda$ of $\rho$ by the following formula \cite[Proposition 5.28]{knapp1996lie}:
\begin{align} \label{eq:casimir_scalar}
c_n^\rho = B(\lambda, \lambda) - B(\lambda, 2 \delta_n),
\end{align}
where $\delta_n$ is the half sum of fundamental roots of $\frk{so}(n)$, which is given by
\begin{align} \label{eq:half_sum}
   \delta_{2m} &= m e_1 + (m-1) e_2 + \ldots + e_m & m \geq 2, \\ 
   \delta_{2m+1} &= \left ( m- \frac{1}{2} \right ) + \left ( m - \frac{3}{2} \right ) e_2 + \ldots + \frac{1}{2} e_m & m \geq 1.
\end{align}
In particular, $2 \delta_n$ always has integer coefficients.

We can now prove:
\begin{lem} \label{lem:optim_poincare}
Let $n \geq 3$. Then, the eigenvalues of the Casimir operator $\Delta^{\SO(n)}$ acting on $L^2(\SO(n))$ for the metric $\langle \cdot, \cdot \rangle = - \frac{1}{2} \mathrm{Tr}$ are integers.
\end{lem}

\begin{proof}
Since $B = (n-2) \langle \cdot, \cdot \rangle$, we see that $\Delta^{\SO(n)} = (n-2) C_{L^2(\SO(n))}$. Then, we decompose $L^2(\SO(n))$ in irreducible subrepresentations using the Peter-Weyl theorem, which are eigenspaces for $C^{L^2}$ by Schur's Lemma. Then, since the highest weights attached to representations of $\SO(n)$ are integral and using Equations \ref{eq:casimir_scalar}, \eqref{eq:half_sum}, we obtain that the eigenvalues of $C^{L^2}$ are integers. It follows that the eigenvalues of $\Delta^{\SO(n)}$ are integers, too.
\end{proof}

Finally, we discuss the first eigenvalue of the Laplacian on the homogeneous spaces $X = \SO(n-1)/K$ listed in Theorem \ref{thm:list_groups}. In general, one can obtain them by first remarking that eigenfunctions on $X$ lift to eigenfunctions on $\SO(n-1)$, and one has to determine to which subrepresentations of $L^2(\SO(n-1))$ these lifts belong.

In two exceptional cases, we have:
\begin{enumerate}
\item For $K = \U(3)$ and the homogeneous space $\SO(6) / \U(3) \cong \mathbb{P}^3(\C)$, we have $\nu = 16$ by \cite[Proposition C.III.1]{berger2006spectre}.
\item For $K = \mathbf{G}_2$ and the homogeneous space $\SO(7) / \mathbf{G}_2 \cong \mathbb{P}^7(\R)$, we have $\nu = 16$ by \cite[Proposition C.II.1]{berger2006spectre}.
\end{enumerate}
Both values are computed using the metric $\langle \cdot, \cdot \rangle$, so we do not have to renormalize them.

\smallskip

We now take $n$ even and consider $K_q \coloneq \SO(q) \times \SO(n-1-q), q \leq (n-2)/2$. It follows from \cite[Theorem 5]{strichartz1975explicit} and \cite[Proposition 5.28]{knapp1996lie} that the highest weights $\lambda$ of representations which arise are of the form $\lambda = (r_1, \ldots, r_q, 0, \ldots, 0)$ in the basis $(\varpi_1, \ldots, \varpi_{(n-2)/2})$ for some nonnegative integers $r_1, \ldots, r_q$ (notice that since $n-1$ is odd, we only have to consider the case $(i)$ of the theorem). Note that for $q \leq q'$, the weights appearing for $K_q$ also appear for $K_{q'}$ (this can be explained in terms of the representations appearing in the theory of the Stiefel manifold $\SO(n) / \SO(n-q)$, see \cite{strichartz1975explicit}). By Equation \eqref{eq:casimir_scalar}, rescaling the metric, we deduce that $\nu(K_q) \leq \nu(K_1) = n-2$, which is the first eigenvalue of the Laplacian on $S^{n-2}$.

Also remark that for $n = 8$, we have $n - 2 = 6 < 16$. As a result, we have, for $n = 7$ or $n$ even,
\begin{align*} 
\max_{K \in \mathcal{G}_n} \nu(K) = 
\begin{cases*}
   16 \text{ if $n = 7, 8$,} \\
   n-2 \text{ if $n \neq 8, 134$ is even,} \\ 
   \max(132, \nu(\mathbf{E}_7 / \Z_2)) \text{ if $n = 134$}.
\end{cases*}
\end{align*}

\printbibliography 

\end{document}